\documentclass[reqno, 11pt]{amsart}
\usepackage{amsthm, amsmath, amssymb, graphicx, tikz,enumerate,paralist,bbm,setspace,bbm}
\usepackage[left=1in,right=1in,top=1in,bottom=1in]{geometry}

\usepackage[color,final]{showkeys} 
\usepackage[colorlinks=false, pdfstartview=FitV, 
pagebackref=false,hidelinks]{hyperref}

\definecolor{refkey}{gray}{.75}
\definecolor{labelkey}{gray}{.5}

\usetikzlibrary{calc,intersections,through,backgrounds,shapes,patterns}
\tikzstyle{p}+=[fill=black, circle, minimum width = 1pt, inner sep =
1pt]

\newtheorem{theorem}{Theorem}[section]
\newtheorem{proposition}[theorem]{Proposition}
\newtheorem{lemma}[theorem]{Lemma}

\newtheorem*{claim}{Claim}
\theoremstyle{definition}
\newtheorem{definition}[theorem]{Definition}
\newtheorem{example}[theorem]{Example}

\theoremstyle{remark}

\newcommand{\abs}[1]{\left\lvert#1\right\rvert}

\newcommand{\norm}[1]{\left\lVert#1\right\rVert}

\newcommand{\set}[1]{\left\{ #1 \right\}}

\newcommand{\ol}{\overline}

\newcommand{\x}{\times}
\newcommand{\e}{\epsilon}
\newcommand{\del}{\partial}

\newcommand{\EE}{\mathbb{E}}
\newcommand{\RR}{\mathbb{R}}

\newcommand{\ZZ}{\mathbb{Z}}

\newcommand{\tg}{\tilde{g}}
\newcommand{\og}{\bar{g}}

\newcommand{\be}{\mathbf{e}}

\newcommand{\bx}{\mathbf{x}}

\title{A relative Szemer\'edi theorem}
\author{David Conlon}
\address{Mathematical Institute\\
Oxford OX1 3LB\\
United Kingdom}
\email{david.conlon@maths.ox.ac.uk}

\author{Jacob Fox}
\address{Department of Mathematics\\
MIT\\
Cambridge\\
MA 02139-4307}
\email{fox@math.mit.edu}

\author{Yufei Zhao}
\address{Department of Mathematics\\
MIT\\
Cambridge\\
MA 02139-4307}
\email{yufeiz@math.mit.edu}

\thanks{The first author was supported by a Royal Society University
  Research Fellowship, the second author was supported by a Simons
  Fellowship, NSF grant DMS-1069197, by an Alfred P. Sloan
Fellowship, and by an MIT NEC Corporation Fund Award, and the third author was
  supported by a Microsoft Research PhD Fellowship}

\begin{document}

\maketitle

\begin{abstract}
The celebrated Green-Tao theorem states that there are arbitrarily long
arithmetic progressions in the primes. One of the main ingredients in their proof is a relative Szemer\'edi theorem which says that any subset of a pseudorandom set of integers of positive relative density contains long arithmetic progressions.

In this paper, we give a simple proof of a strengthening of the relative Szemer\'edi
theorem, showing that a much weaker pseudorandomness condition is sufficient. Our
strengthened version can be applied to give the first relative Szemer\'edi
theorem for $k$-term arithmetic progressions in pseudorandom subsets of
$\mathbb{Z}_N$ of density $N^{-c_k}$.

The key component in our proof is an extension of the regularity method to sparse
pseudorandom hypergraphs, which we believe to be interesting in its own right. From this we derive a relative extension of the hypergraph removal lemma. This is a strengthening of an
earlier theorem used by Tao in his proof that the Gaussian primes contain arbitrarily shaped constellations and, by standard arguments, allows us to deduce the relative Szemer\'edi theorem.
\end{abstract}

\section{Introduction} \label{sec:intro}

The Green-Tao theorem \cite{GT08} states that the primes contain arbitrarily long arithmetic progressions. This result, along with their subsequent work \cite{GT10} on determining the asymptotics for the number of prime $k$-tuples in arithmetic progression, constitutes one of the great breakthroughs in 21st century mathematics.

The proof of the Green-Tao theorem has two key steps. The first step,
which Green and Tao refer to as the ``main new ingredient" of their
proof, is to establish a relative Szemer\'edi theorem. Szemer\'edi's
theorem \cite{Sze75} states that any dense subset of the integers
contains arbitrarily long arithmetic progressions. More formally, we
have the following theorem, which is stated for $\ZZ_N := \ZZ / N \ZZ$
but easily implies an equivalent statement in the set $[N] :=\{1, 2, \dots, N\}$.

\begin{theorem}[Szemer\'edi's theorem] \label{thm:intro-sz} For every
  natural number $k \geq 3$ and every $\delta > 0$, as long as $N$ is
  sufficiently large, any subset of $\ZZ_N$ of density at least
  $\delta$ contains an arithmetic progression of length $k$.
\end{theorem}

A relative Szemer\'edi theorem is a similar statement where the ground set is no longer the set $\ZZ_N$ but rather a sparse pseudorandom subset of $\ZZ_N$.

The second step in their proof is to show that the primes are a dense subset of a pseudorandom set of ``almost primes", sufficiently pseudorandom that the relative Szemer\'edi theorem holds. Then, since the primes are a dense subset of this pseudorandom set, an application of the relative Szemer\'edi theorem implies that the primes contain arbitrarily long arithmetic progressions. This part of the proof use some ideas from the work of Goldston and Y{\i}ld{\i}r{\i}m~\cite{GY03} (and was subsequently simplified in~\cite{Taonote}).

In the work of Green and Tao, the pseudorandomness conditions on the
ground set are known as the linear forms condition and the correlation
condition. Roughly speaking, both of these conditions say that, in
terms of the number of solutions
to certain linear systems of equations, the set behaves like a random
set of the same density. A natural question is whether these pseudorandomness conditions can be weakened. We address this question by giving a simple proof for a strengthening of the relative Szemer\'edi theorem, showing that a weak linear forms condition is sufficient for the theorem to hold.

This improvement has two aspects. We remove the correlation condition entirely but we also reduce the set of linear forms for which the correct count is needed. In particular, we remove those corresponding to the dual function condition, a pointwise boundedness condition stated explicitly by Tao \cite{Tao06jam} in his work on constellations in the Gaussian primes but also used implicitly in \cite{GT08}.

To state the main theorem, we will assume the definition of the $k$-linear forms condition. The formal definition, which may be found in Section~\ref{sec:definitions-results} below, is stated for measures rather than sets but we will ignore this relatively minor distinction here, reserving a more complete discussion of our terminology for there.

\begin{theorem} [Relative Szemer\'edi
  theorem] \label{thm:RelSzemIntro} For every natural number $k \geq
  3$ and every $\delta > 0$, if $S \subset \ZZ_N$ satisfies the
  $k$-linear forms condition and $N$ is sufficiently large, then any
  subset of $S$ of relative density at least $\delta$ contains an arithmetic
  progression of length $k$.
\end{theorem}

One of the immediate advantages of this theorem is that it simplifies the proof of the Green-Tao theorem. In addition to giving a simple proof of the relative Szemer\'edi theorem, it removes the need for the number-theoretic estimates involved in establishing the correlation condition for the almost primes. A further advantage is that, by removing the correlation condition, the relative Szemer\'edi theorem now applies to pseudorandom subsets of $\mathbb{Z}_N$ of density $N^{-c_k}$. With the correlation condition, one could only hope for such a theorem down to densities of the form $N^{-o(1)}$.


While the relative Szemer\'edi theorem is the main result of this paper, the main advance is an approach to regularity in sparse pseudorandom hypergraphs. This allows us to prove analogues of several well-known combinatorial theorems relative to sparse pseudorandom hypergraphs. In particular, we prove a sparse analogue of the hypergraph removal lemma. It is from this that we derive our relative Szemer\'edi theorem. As always, applying the regularity method has two steps, a regularity lemma and a counting lemma. We provide novel approaches to both.

A counting lemma for subgraphs of sparse pseudorandom graphs was already proved by the authors in \cite{CFZ14}. In this paper, we simplify and streamline the approach taken there in order to prove a counting lemma for subgraphs of sparse pseudorandom hypergraphs. This result is the key technical step in our proof and, perhaps, the main contribution of this paper. Apart from the obvious difficulties in passing from graphs to hypergraphs, the crucial difference between this paper and \cite{CFZ14} is in the type of pseudorandomness considered. For graphs, we have a long-established notion of pseudorandomness known as jumbledness. The greater part of \cite{CFZ14} is then concerned with optimizing the jumbledness condition which is necessary for counting a particular graph $H$. For hypergraphs, we use an analogue of the linear forms condition first considered by Tao \cite{Tao06jam}. It says that our hypergraph is pseudorandom enough for counting $H$ within subgraphs if it contains asymptotically the correct count for the $2$-blow-up of $H$ and all its subgraphs.

We also use an alternative approach to regularity in sparse
hypergraphs. While it would be natural to use a sparse hypergraph
regularity lemma (and, following our approach in \cite{CFZ14}, this was
how we initially proceeded), it suffices to use a weak sparse
hypergraph regularity lemma which is an extension of the weak
regularity lemma of Frieze and Kannan~\cite{FK99}. This is also closely
related to the transference theorem used by Green and Tao (see, for
example,~\cite{Gow10} or \cite{RTTV08, TTV09}, where it is also referred to as the dense model theorem).

With both a regularity lemma and a counting lemma in place, it is then
a straightforward matter to prove a relative extension of the famous
hypergraph removal lemma \cite{Gow07, NRS06, RS04, RS06, Tao06jcta}. Such a
theorem was first derived by Tao \cite{Tao06jam} in his work on
constellations in the Gaussian primes but, like the Green-Tao relative
Szemer\'edi theorem, needs both a correlation condition and a dual
function condition.\footnote{The problem of relative hypergraph
  removal was also recently considered by Towsner~\cite{Tow}.} Our approach removes these conditions. The final step in the proof of the relative Szemer\'edi theorem is then a standard reduction used to derive Szemer\'edi's theorem from the hypergraph removal lemma. The details of this reduction already appear in \cite{Tao06jam} but we include them here for completeness. In fact, the paper is self-contained apart from assuming the hypergraph removal lemma.

In Section~\ref{sec:definitions-results}, we state our results including the relative
Szemer\'edi theorem and the removal, regularity, and counting lemmas. In
Section~\ref{sec:relative-szemeredi}, we deduce the relative
multidimensional Szemer\'edi theorem from our relative hypergraph
removal lemma. In Section~\ref{sec:removal-lemma}, we prove the removal
lemma from the regularity and counting lemmas. We will prove our weak sparse hypergraph regularity lemma in Section~\ref{sec:weak-regularity} and the associated counting lemma in Section~\ref{sec:counting-lemma}. We conclude, in Section~\ref{sec:concluding}, with some remarks.

\section{Definitions and results} \label{sec:definitions-results}

\subsection*{Notation} \emph{Dependence on $N$.} We consider functions
$\nu = \nu^{(N)}$, where $N$ (usually suppressed) is assumed to be
some large integer. We write $o(1)$ for a quantity that tends to zero
as $N \to \infty$.

\noindent
\emph{Expectation.}  We write $\EE[f(x_1, x_2, \dots) \vert x_1 \in
A_1, x_2 \in A_2,\dots]$ for the expectation of $f(x_1, x_2, \dots)$ when
each $x_i$ is chosen uniformly and independently at random from $A_i$.

\subsection{A relative Szemer\'edi theorem}

Here is an equivalent weighted version of Szemer\'edi's theorem as formulated, for example, in \cite[Prop.~2.3]{GT08}.

\begin{theorem}
  [Szemer\'edi's theorem, weighted version]
  \label{thm:sz-weighted}
  For every $k \geq 3$ and $\delta > 0$, there exists $c > 0$ such that
  for $N$ sufficiently large and any nonnegative function $f \colon \ZZ_N \to [0,1]$ satisfying
  $\EE[f] \geq \delta$,
  \begin{equation} \label{eq:sz-weighted}
  \EE[ f(x)f(x+d)f(x + 2d) \cdots f(x+(k-1)d) \vert x,d \in \ZZ_N] \geq c.
\end{equation}
\end{theorem}

A relative Szemer\'edi theorem would instead ask for the nonnegative function $f$
to be bounded above by a measure $\nu$ instead of the constant function $f$.
For us, a measure will be any nonnegative function on $\ZZ_N$. We do not
explicitly assume the additional condition that
\[\EE[\nu(x) \vert x \in \ZZ_N] = 1 + o(1),\]
but this property follows from the linear forms condition that we will now assume.
Such measures are more
general than subsets, as any subset $S \subseteq
\ZZ_N$ (e.g.,~ in Theorem~\ref{thm:RelSzemIntro}) can be thought of as a
measure on $\ZZ_N$ taking value $N/\abs{S}$ on $S$ and 0
elsewhere. The dense case, as in Theorem~\ref{thm:sz-weighted}, corresponds to taking $\nu = 1$.
Our notion of pseudorandomness for measures $\nu$ on $\ZZ_N$ is now as follows.

\begin{definition}[Linear forms condition]
  \label{def:Z-lfc}
  A nonnegative function $\nu = \nu^{(N)} : \mathbb{Z}_N
  \rightarrow \mathbb{R}_{\geq 0}$ is said to obey the
  \emph{$k$-linear forms condition} if one has
  \begin{equation}
    \label{eq:Z-lfc}
    \EE \Bigl[
    \prod_{j=1}^k \prod_{\omega \in \{0,1\}^{[k]\setminus\{j\}}} \nu
    \Bigl( \sum_{i=1}^k (i-j) x_i^{(\omega_i)} \Bigr)^{n_{j,\omega}}
    \Big\vert
    x_1^{(0)}, x_1^{(1)}, \dots, x_k^{(0)},x_k^{(1)} \in \ZZ_N
    \Bigr]  = 1 + o(1)
  \end{equation}
  for any choices of exponents $n_{j,\omega} \in \{0,1\}$.
\end{definition}

\begin{example}
  \label{ex:Z-lfc}
  For $k = 3$, condition~\eqref{eq:Z-lfc} says that
  \begin{multline*}
    \EE[\nu(y+2z) \nu(y'+2z) \nu(y+2z') \nu(y'+2z')
    \nu(-x+z) \nu(-x'+z) \nu(-x+z') \nu(-x'+z') \\
    \cdot \nu(-2x-y) \nu(-2x'-y) \nu(-2x-y') \nu(-2x'-y')
    \vert x,x',y,y',z,z' \in \ZZ_N] = 1 + o(1)
  \end{multline*}
  and similar conditions hold if one or more of the twelve $\nu$ factors in the
  expectation are erased.
\end{example}

Our linear forms condition is much weaker that that used in
Green and Tao~\cite{GT08}. In particular, Green and Tao need
to assume that pointwise estimates such as
\[
\EE[\nu(a + x)\nu(a + y) \nu(a+x+y) \vert x,y \in \ZZ_N] = 1 + o(1)
\]
hold uniformly over all $a \in \ZZ_N$. Such linear forms do not arise in our proof. Moreover, to
prove their relative Szemer\'edi theorem, Green and Tao
need to assume a further pseudorandomness
condition, which they call the correlation condition. This condition
also does not arise in our proofs.
Indeed, we prove that a relative Szemer\'edi
theorem holds given only the linear forms condition defined above.

\begin{theorem}
  [Relative Szemer\'edi theorem]
  \label{thm:rel-sz}
  For every $k \geq 3$ and $\delta > 0$, there exists $c > 0$ such that if $\nu \colon \ZZ_N \to \RR_{\geq 0}$ satisfies the $k$-linear
  forms condition, $N$ is sufficiently large, and $f \colon \ZZ_N \to \RR_{\geq 0}$ satisfies $0
  \leq f(x) \leq \nu(x)$ for all $x \in \ZZ_N$ and $\EE[f] \geq
  \delta$, then
  \begin{equation} \label{eq:rel-sz}
  \EE[ f(x)f(x+d)f(x + 2d) \cdots f(x+(k-1)d) \vert x,d \in \ZZ_N] \geq c.
\end{equation}
\end{theorem}

We note that both here and in Theorem \ref{thm:RelSzemIntro}, the phrase ``$N$ is sufficiently large" indicates not only a dependency on $\delta$ and $k$ as in the usual version of Szemer\'edi's theorem but also a dependency on the $o(1)$ term in the linear forms condition. We will make a similar assumption in many of the theorems stated below.

We prove Theorem~\ref{thm:rel-sz} using a new relative hypergraph
removal lemma.\footnote{Green and Tao~\cite{GT08} prove a
transference result that allows them to apply the dense version of
Szemer\'edi's theorem as a black box. This allows them to show that the optimal $c$ in
\eqref{eq:rel-sz} can be taken to be the same as the optimal $c$ in
\eqref{eq:sz-weighted}. The proof in this paper goes through the hypergraph
removal lemma and thus does not obtain the same $c$. Nevertheless,
one can obtain our result with the same $c$ by modifying
the argument to an arithmetic setting, as done by the third author
in a follow-up paper~\cite{Zhao14}.} In the next subsection we set up the notation for
hypergraphs and state the corresponding pseudorandomness
hypothesis.

\subsection{Hypergraphs}

We borrow most of our notation and definitions from Tao~\cite{Tao06jam, Tao06jcta}.

\begin{definition}
  [Hypergraphs] \label{def:hypergraphs}
  Let $J$ be a finite set and $r > 0$. Define $\binom{J}{r} = \set{e
    \subseteq J : \abs{e} = r}$ to be the set of all $r$-element subsets
  of $J$. An \emph{$r$-uniform hypergraph} on $J$ is defined to be any
  subset $H \subseteq \binom{J}{r}$.
\end{definition}

\begin{definition}
  [Hypergraph system] \label{def:hypergraph-system}
  A \emph{hypergraph system} is a quadruple $V =
  (J, (V_j)_{j \in J}, r, H)$, where $J$ is a finite set, $(V_j)_{j
    \in J}$ is a collection of finite non-empty sets indexed by $J$,
  $r \geq 1$ is a positive integer, and $H \subseteq \binom{J}{r}$ is an
  $r$-uniform hypergraph. For any $e \subseteq J$, we set $V_e :=
  \prod_{j \in e} V_j$. For any $x = (x_j)_{j \in J} \in V_J$ and any subset $J'
  \subseteq J$, we write $x_{J'} = (x_j)_{j \in J'} \in V_{J'}$ to mean the natural
  projection of $x$ onto the coordinates $J'$. Finally, for any $e \subseteq J$, we
  write $\del e$ for the set $\{f \subsetneq e : |f| = |e| - 1\}$, the
  skeleton of $e$.
\end{definition}

\begin{definition}
  [Weighted hypergraphs] \label{def:weighted-hypergraphs}
  Let $V = (J, (V_j)_{j \in J}, r, H)$ be a hypergraph system. A
  \emph{weighted hypergraph} on $V$ is a collection $g = (g_e)_{e \in
    H}$ of functions $g_e \colon V_e \to \RR_{\geq 0}$ indexed by $H$. We write 0 and 1 to
  denote the constant-valued weighted hypergraphs of uniform weight
  $0$ and $1$, respectively.  Given two
  weighted hypergraphs $g$ and $\nu$ on the same hypergraph system, we
  write $g \leq
  \nu$ to mean that $g_e \leq \nu_e$ for all $e$, which in turn means
  that $g_e(x_e) \leq \nu_e(x_e)$ for all $x_e \in V_e$.
\end{definition}

The weighted hypergraph $\nu$ plays an analogous role to the $\nu$ in
Theorem~\ref{thm:rel-sz}, with $\nu = 1$ again corresponding to the
dense case. We have an analogous linear forms condition for $\nu$ as a
weighted hypergraph. We use the following
indexing notation. For a finite set $e$ and $\omega \in \{0,1\}^e$, we
write $x^{(\omega)}_e$ to mean the tuple $(x^{(\omega_j)}_j)_{j \in
  e}$. We also write $x^{(0)}_e := (x^{(0)}_j)_{j \in e}$ and
similarly with $x^{(1)}_e$.

\begin{definition}
  [Linear forms condition] \label{def:H-lfc}
  A weighted hypergraph
$\nu = \nu^{(N)}$ on the hypergraph system $V = V^{(N)} = (J,
(V_j^{(N)})_{j \in J}, r, H)$ is said to obey the \emph{$H$-linear forms
  condition} (or simply the \emph{linear forms condition} if there is no confusion) if one has
\begin{equation}
  \label{eq:H-lfc}
      \EE\Bigl[\prod_{e \in H} \prod_{\omega \in \{0,1\}^e}
    \nu_e(x_e^{(\omega)})^{n_{e,\omega}}
    \Big\vert
    x_J^{(0)}, x_J^{(1)} \in
    V_J
    \Bigr]
    = 1 + o(1)
\end{equation}
for any choices of exponents $n_{e,\omega} \in \{0,1\}$.
\end{definition}

\begin{example}
  \label{ex:linear-forms}
  Let $H$ be the set of all pairs in $J = \{1,2,3\}$.
  The linear forms condition says that
\[
\EE\Bigl[
  \prod_{ij=12,13,23}
  \nu_{ij}(x_i,x_j) \nu_{ij}(x'_i,x_j) \nu_{ij}(x_i,x'_j)
  \nu_{ij}(x'_i,x'_j)
  \Big\vert x_1,x'_1 \in V_1,\
  x_2,x'_2\in V_2,\ x_3,x'_3 \in V_3
  \Bigr]= 1 + o(1)
\]
and similarly if one or more of the twelve $\nu$ factors are
deleted. This expression represents the weighted homomorphism density of
$K_{2,2,2}$ in the weighted tripartite graph given by $\nu$, as illustrated in
  Figure~\ref{fig:K3}(b) (the vertices of $K_{2,2,2}$ must map into
  the corresponding parts). Deleting some $\nu$ factors corresponds to
  considering various subgraphs of $K_{2,2,2}$, e.g., Figure~\ref{fig:K3}(c).
\end{example}

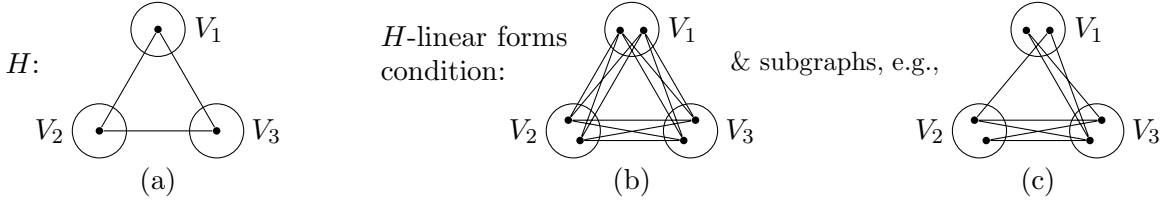
\begin{figure}
\begin{tikzpicture}[scale=.45]
  \begin{scope}[shift={(-1,0)}]
    \node at (-4,1) {$H$:};
  \node at (0,-2.5) {(a)};
  \draw (90:2) circle (.8); \node at (1.5,2) {$V_1$};
  \draw (210:2) circle (.8); \node at (-3.2,-1) {$V_2$};
  \draw (330:2) circle (.8);\node at (3.2,-1) {$V_3$};
  \node[p] (a) at (90:2) {};
  \node[p] (b) at (210:2) {};
  \node[p] (c) at (330:2) {};
  \draw (a)--(b)--(c)--(a);
\end{scope}

  \begin{scope}[shift={(13,0)}]
    \node[text width = 8em] at (-4,1.3) {$H$-linear forms condition:};
    \node[font=\small] at (6,1) {\& subgraphs, e.g.,};
  \node at (0,-2.5) {(b)};
  \draw (90:2) circle (.8); \node at (1.5,2) {$V_1$};
  \draw (210:2) circle (.8); \node at (-3.2,-1) {$V_2$};
  \draw (330:2) circle (.8);\node at (3.2,-1) {$V_3$};
  \node[p] (a1) at (80:2) {};
  \node[p] (a2) at (100:2) {};
  \node[p] (b1) at (200:2) {};
  \node[p] (b2) at (220:2) {};
  \node[p] (c1) at (320:2) {};
  \node[p] (c2) at (340:2) {};
  \draw (a1)--(b1)--(a2)--(b2)--(a1)
  (a1)--(c1)--(a2)--(c2)--(a1)
  (c1)--(b1)--(c2)--(b2)--(c1);
\end{scope}

  \begin{scope}[shift={(25,0)}]
  \node at (0,-2.5) {(c)};
  \draw (90:2) circle (.8); \node at (1.5,2) {$V_1$};
  \draw (210:2) circle (.8); \node at (-3.2,-1) {$V_2$};
  \draw (330:2) circle (.8);\node at (3.2,-1) {$V_3$};
  \node[p] (a1) at (80:2) {};
  \node[p] (a2) at (100:2) {};
  \node[p] (b1) at (200:2) {};
  \node[p] (b2) at (220:2) {};
  \node[p] (c1) at (320:2) {};
  \node[p] (c2) at (340:2) {};
  \draw (a1)--(b1)
  (a1)--(c1)--(a2)--(c2)
  (c1)--(b1)--(c2)--(b2)--(c1);
\end{scope}

\end{tikzpicture}
\vspace{-1em}
\caption{Linear forms conditions for $H = K_3$. See Example~\ref{ex:linear-forms}.} \label{fig:K3}
\end{figure}

  In general, the $H$-linear forms condition says that $\nu$ has
   roughly the expected density for the 2-blow-up\footnote{By the 2-blow-up of $H$
  we mean the hypergraph consisting of
  vertices $j^{(0)}, j^{(1)}$ for each $j \in J$, and edges
  $e^{(\omega)} := \{j^{(\omega_j)}: j \in e\}$ for any $e \in H$ and $\omega \in \{0,1\}^e$. We actually do not need
  the full strength of this assumption. It suffices to assume that
  $\nu$ has roughly the expected density for any subgraph of a \emph{weak} 2-blow-up of
  $H$, where by a weak 2-blow-up we mean the following. Fix some edge
  $e_1 \in H$ (we will need to assume the condition for all
  $e_1$). The weak 2-blow-up of $H$ with respect to $e_1$ is the
  subgraph of the usual $2$-blow-up consisting of all edges $e^{(\omega)}$ where
  $\omega_i = \omega_j$ for any $i,j \in e \setminus e_1$. This weaker
  version of the $H$-linear forms condition is all we shall use for
  the proof, although everything to follow will be stated as in
  Definition~\ref{def:H-lfc} for clarity.} of $H$
  as well as any subgraph of the 2-blow-up.
  Our linear forms condition for hypergraphs coincides with the one used by
  Tao~\cite[Def.~2.8]{Tao06jam}, although in \cite{Tao06jam} one
  assumes additional pseudorandomness hypotheses on $\nu$ known as the
  dual function condition and the correlation condition.

\subsection{Hypergraph removal lemma}

The hypergraph removal lemma was first proved by Gowers \cite{Gow07}
and by Nagle, R\"odl, Schacht, and Skokan \cite{NRS06,RS04,RS06}. It
states that for every $r$-uniform hypergraph $H$ on $h$ vertices,
every $r$-uniform hypergraph on $n$ vertices with $o(n^h)$ copies of
$H$ can be made $H$-free by removing $o(n^r)$ edges. As first
explicitly stated and proved by Tao \cite{Tao06jcta}, the proof of the
hypergraph removal lemma further gives that the edges can be removed
in a low complexity way (this idea will soon be made formal).  We will
use a slightly stronger version, where edges are given weights in
the interval $[0,1]$. This readily follows from the usual
version by a simple rounding argument, as done in
\cite[Thm.~3.7]{Tao06jam}. We state this result as Theorem
\ref{thm:dense-hypergraph-removal} below.

\begin{definition} \label{def:complexity-set}
  For any set $e$ of size $r$ and any $E_e \subseteq V_e = \prod_{j \in e}
  V_j$, we define the \emph{complexity} of $E_e$ to be the minimum
  integer $T$ such that there is a partition of $E_e$ into $T$ sets
  $E_{e,1}, \dots, E_{e,T}$, so that each $E_{e,i}$ is the set of
  $r$-cliques of some $(r-1)$-uniform hypergraph, meaning that there
  exists some $B_{f,i} \subseteq V_f$ for each $f \in \del e$ so that
  $1_{E_{e,i}}(x_e) = \prod_{f \in \del e} 1_{B_{f,i}}(x_f)$ for all
  $x_e \in V_e$.
\end{definition}

\begin{theorem}[Weighted hypergraph removal
  lemma]\label{thm:dense-hypergraph-removal} For every $\e > 0$ and finite set $J$,
  there exists $\delta > 0$ and $T > 0$ such that the following holds. Let $V=(J,(V_j)_{j \in J},r,H)$
  be a hypergraph system. Let $g$ be a weighted hypergraph on $V$
  satisfying $0 \leq g \leq 1$ and
  \[
  \EE\Bigl[ \prod_{e \in H} g_e(x_e) \Big\vert x \in V_J \Bigr]
  \leq \delta.
  \]
  Then for each $e \in H$ there exists a
  set $E_e' \subseteq V_e$ for which $V_e \setminus E'_e$ has
  complexity at most $T$ and such that
  \[
  \prod_{e \in H} 1_{E'_e}(x_e) = 0 \text{ for all } x \in V_J
  \]
  and for all $e \in H$ one has
  \[
  \EE[g_e(x_e)1_{V_e \setminus E'_e}(x_e) \vert x_e \in V_e] \leq \e.
  \]
\end{theorem}

We prove a relativized extension of the hypergraph removal lemma. A
relative hypergraph removal lemma was already established by Tao in \cite{Tao06jam}, where he assumed the majorizing measure satisfies three conditions: the linear forms condition, the correlation condition, and the dual function condition. We again show that a linear forms condition is sufficient.

\begin{theorem}[Relative hypergraph removal lemma]
  \label{thm:rel-removal}
  For every $\e > 0$ and finite set $J$, there exists $\delta > 0$ and $T > 0$ such
  that the following holds.
  Let $V=(J,(V_j)_{j \in J},r,H)$ be a hypergraph system. Let $\nu$ and $g$
  be weighted hypergraphs on $V$. Suppose $0 \leq g \leq \nu$,
  $\nu$ satisfies the $H$-linear forms condition, and $N$ is sufficiently large. If
  \[
  \EE\Bigl[ \prod_{e \in H} g_e(x_e) \Big\vert x \in V_J \Bigr]
  \leq \delta,
  \]
then for each $e \in H$ there exists a set
$E_e' \subseteq V_e$ for which $V_e \setminus E'_e$ has complexity
at most $T$ and such that
\[
\prod_{e \in H}1_{E'_e}(x_e) = 0  \text{ for all } x \in V_J
\]
and for all $e \in H$ one has
\[
\mathbb{E}[g_e(x_e)1_{V_e \setminus E'_e}(x_e)|x_e \in V_e] \leq \e.
\]
\end{theorem}

In Section~\ref{sec:removal-lemma} we will deduce Theorem~\ref{thm:rel-removal} from
Theorem~\ref{thm:dense-hypergraph-removal} by applying the weak regularity
lemma and the counting lemma which are stated in the next two subsections.

\subsection{Weak hypergraph regularity}

The Frieze-Kannan weak regularity lemma~\cite{FK99} allows one to
approximate in cut-norm a matrix (or graph) with entries in the
interval $[0,1]$ by another matrix of low complexity. A major
advantage over simply applying Szemer\'edi's regularity lemma is that
the complexity has only an exponential dependence on the approximation
parameter, as opposed to the tower-type bound that is incurred by
Szemer\'edi's regularity lemma. Unfortunately, these regularity lemmas
are not meaningful for sparse graphs as the error term is too large in
this setting. Following sparse extensions of Szemer\'edi's regularity
lemma by Kohayakawa~\cite{Koh97} and R\"odl, a sparse extension of the
weak regularity lemma  was proved by Bollob\'as and Riordan \cite{BR09} and by Coja-Oghlan, Cooper, and
Frieze~\cite{CCF09}. In \cite{CCF09}, they further generalize this to $r$-dimensional
tensors (or $r$-uniform hypergraphs), but it only gives an
approximation which is close in density on all hypergraphs induced by
large vertex subsets. In order to prove a relative hypergraph removal
lemma, we will need a stronger approximation, which is close in
density on all dense $r$-uniform hypergraphs formed by the
clique set of some $(r-1)$-uniform hypergraph. In Section~\ref{sec:weak-regularity},
we will prove a more general sparse regularity lemma. For now, we state the result in the form that we need.

The weak regularity lemma approximates a weighted hypergraph $g$ on
$V$ by another weighted hypergraph $\tg$ of bounded complexity which satisfies $0 \leq \tg \leq
1$. One can think of $\tg$ as a dense approximation of $g$. The
following definition makes precise in what sense $\tg$ approximates $g$.

\begin{definition}
  [Discrepancy pair] \label{def:disc-pair}
  Let $e$ be a finite set and $g_e, \tg_e \colon \prod_{j \in e} V_j
  \to \RR_{\geq 0}$ be two nonnegative functions. We say that
  $(g_e,\tg_e)$ is an \emph{$\e$-discrepancy pair} if for
  all subsets $B_f \subseteq V_f$, $f \in \del e$, one has
  \begin{equation}
    \label{eq:disc-pair}
    \Bigl\lvert
    \EE\Bigl[
    (g_e(x_e) - \tg_e(x_e)) \prod_{f \in \del e}
        1_{B_f}(x_f)
        \Big\vert x_e \in V_e
        \Bigr]\Bigr\rvert \leq \e.
  \end{equation}
  For two weighted hypergraphs $g$ and $\tg$ on $(J, (V_j)_{j \in J},
  r, H)$, we say that $(g, \tg)$ is an \emph{$\e$-discrepancy pair} if
  $(g_e, \tg_e)$ is an $\e$-discrepancy pair for all $e \in H$.
\end{definition}

One needs an
additional hypothesis on $g$ in order to prove a weak regularity
lemma. The condition roughly says that $g$ contains ``no dense spots.''

\begin{definition}
  [Upper regular] \label{def:upper-regular}
  Let $e$ be a finite set, $g_e \colon \prod_{j \in e} V_j
  \to \RR_{\geq 0}$ a nonnegative function, and $\eta > 0$. We say that $g_e$ is
  \emph{upper $\eta$-regular} if for all subsets $B_f \subseteq V_f$,
  $f \in \del e$, one has
    \begin{equation}
    \label{eq:upper-regular}
    \EE\Bigl[
    (g_e(x_e) - 1) \prod_{f \in \del e}
        1_{B_f}(x_f)
        \Big\vert x_e \in V_e
        \Bigr] \leq \eta.
  \end{equation}
A hypergraph $g$ on  on $(J, (V_j)_{j \in J}, r, H)$ is upper $\eta$-regular if $g_e$ is upper $\eta$-regular for all $e \in H$.
\end{definition}

Note that unlike \eqref{eq:disc-pair}, there is no absolute value on
the left-hand side of \eqref{eq:upper-regular}. The upper regularity
hypothesis is needed for establishing the sparse regularity
lemma. Fortunately, this mild
hypothesis is automatically satisfied in our setting. We will say more
about this in Section  \ref{subsect:stronglinearform}.

\begin{lemma} \label{lem:nu-upper-regular}
Let $V=(J,(V_j)_{j \in J},r,H)$ be a hypergraph system. Let $\nu$ and $g$ be
weighted hypergraphs on $V$. Suppose $0 \leq g \leq \nu$ and $\nu$
satisfies the $H$-linear forms condition. Then $g$ is upper $o(1)$-regular.
\end{lemma}



Define the \emph{complexity} of a function $g \colon V_e \to [0,1]$ to
be the minimum $T$ such that there is a partition of
$V_e$ into $T$ subgraphs $S_1, \dots S_T$, each of which is the set
of $r$-cliques of some $(r-1)$-uniform hypergraph (see Definition~\ref{def:complexity-set}), and such that $g$
is constant on each $S_i$. We state the regularity lemma below with a
complexity bound on $\tg$, although the complexity bound will not actually be needed for
our application.

\begin{theorem}[Sparse weak regularity lemma] \label{thm:weak-reg-hyp}
  For any $\e > 0$ and function $g \colon V_1 \x \cdots \x V_r
  \to \RR_{\geq 0}$ which is upper $\eta$-regular with $\eta \leq 2^{-40r/\e^2}$, there exists $\tg
  \colon V_1 \x \cdots \x V_r \to [0,1]$ with complexity at most $2^{20r/\e^2}$ such that $(g,\tg)$ is an
  $\e$-discrepancy pair.
\end{theorem}

\noindent
The special case $r=2$ is the sparse extension of the Frieze-Kannan
weak regularity lemma.

\subsection{Counting lemma}

Informally, the counting lemma says that if $(g,\tg)$ is an
$\e$-discrepancy pair, with the additional assumption that $g \leq
\nu$ and $\tg \leq 1$, then the density of $H$ in $\tg$ is close to
the density of $H$ in $g$. This sparse counting lemma is perhaps the most
novel ingredient in this paper.

\begin{theorem}[Counting lemma]
  \label{thm:counting-lemma}
  For every $\gamma > 0$ and finite set $J$, there exists an $\e > 0$
  so that the following holds.
  Let $V = (J, (V_j)_{j \in J}, r, H)$ be a hypergraph system and
  $\nu$, $g$, $\tg$ be weighted hypergraphs on $V$. Suppose that $\nu$
  satisfies the $H$-linear forms condition and $N$ is sufficiently
  large. Suppose also that $0 \leq g
  \leq \nu$, $0 \leq \tg \leq 1$, and
  $(g,\tg)$ is an $\e$-discrepancy pair. Then
  \begin{equation}
    \label{eq:counting-lemma}
    \Bigl\vert \EE\Bigl[\prod_{e \in H} g_e(x_e)
        \Big\vert
        x \in V_J
        \Bigr] -
        \EE\Bigl[\prod_{e \in H} \tg_e(x_e)
        \Big\vert
        x \in V_J
        \Bigr]\Bigr\rvert \leq \gamma.
  \end{equation}
\end{theorem}

As a corollary, Theorem~\ref{thm:counting-lemma} also holds if the
hypothesis $0 \leq
\tg \leq 1$ is replaced by $0 \leq \tg \leq \nu$. Indeed, we can use
the weak regularity lemma, Theorem~\ref{thm:weak-reg-hyp}, to find a
common $1$-bounded approximation to $g$ and $\tg$. The result then follows from Theorem~\ref{thm:counting-lemma}
and the triangle inequality.

To summarize, to get a counting lemma for a fixed hypergraph $H$ in a subgraph of a
pseudorandom host hypergraph, it suffices to know that the host
hypergraph has approximately the expected count for a somewhat larger
family of hypergraphs (namely, subgraphs of the 2-blow-up of $H$).

\section{The relative Szemer\'edi theorem} \label{sec:relative-szemeredi}

In this section, we deduce the relative Szemer\'edi theorem,
Theorem~\ref{thm:rel-sz}, from the relative hypergraph removal lemma, Theorem~\ref{thm:rel-removal}. We
use the relative hypergraph removal lemma to prove a relative arithmetic
removal lemma, Theorem~\ref{thm:rel-mul-sz-removal}. This result then easily implies a relative version of the
multidimensional Szemer\'edi theorem of Furstenberg and Katznelson~\cite{FK78}. This is Theorem~\ref{thm:rel-mul-sz} below. The relative Szemer\'edi
theorem, Theorem~\ref{thm:rel-sz}, follows as a special case of Theorem~\ref{thm:rel-mul-sz} by setting $Z = Z' = \ZZ_N$ and
$\phi_j(d) = (j-1)d$. One may easily check that the linear forms condition for the
resulting hypergraph is satisfied if $\nu \colon \ZZ_N \to \RR_{\geq 0}$
satisfies the $k$-linear forms condition.


The statement and proof of Theorem~\ref{thm:rel-mul-sz} closely
follows the write-up in Tao
\cite [Thm~2.18]{Tao06jam}, adapted in a straightforward way to our
new pseudorandomness conditions as well as to the
slightly more general setting of functions instead of subsets. Earlier versions of this type of
argument for deducing Szemer\'edi-type results (in the dense setting) from graph and hypergraph
removal lemmas were given by Ruzsa and Szemer\'edi~\cite{RS78},
Frankl and R\"odl~\cite{FR02}, and Solymosi~\cite{Sol04,Sol03}.

\begin{theorem}[Relative multidimensional Szemer\'edi theorem]
  \label{thm:rel-mul-sz} For a finite set $J$ and $\delta > 0$, there
  exists $c > 0$ so that the following holds. Let $Z, Z'$ be two finite additive groups
  and let $(\phi_j)_{j \in J}$ be a finite collection of group
  homomorphisms $\phi_j: Z \to Z'$ from $Z$
  to $Z'$.  Assume that the elements $\{
  \phi_i(d) - \phi_j(d): i,j \in J, d \in Z\}$ generate $Z'$ as an
  abelian group.  Let $\nu: Z' \to \RR_{\geq 0}$ be a nonnegative
  function with the property that in the hypergraph system $V = (J,
  (V_j)_{j \in J}, r,H)$, with $V_j := Z$, $r := |J|-1$, and $H := {J
    \choose r}$, the weighted hypergraph $(\nu_e)_{e \in H}$ defined
  by
\[
\nu_{J \backslash \{j\}}((x_i)_{i \in J\setminus\{j\}}) := \nu\Bigl(\sum_{i \in J\setminus\{j\}}
(\phi_i(x_i) - \phi_j(x_i))\Bigr)
\]
satisfies the $H$-linear forms condition. Assume that $N$ is sufficiently
large. Then, for any $f \colon Z' \to
\RR_{\geq 0}$ satisfying $0 \leq f(x) \leq
\nu(x)$ for all $x \in Z'$ and $\EE[f] \geq \delta$,
\begin{equation}\label{eq:rel-mul-sz-small-a}
\EE \Bigl[ \prod_{j \in J} f(a + \phi_j(d)) \Big\vert
a \in Z', d \in Z \Bigr] \geq c.
\end{equation}
\end{theorem}

\begin{example} \label{ex:corner}
  Let $S
\subset \ZZ_N \x \ZZ_N$. Suppose the associated measure $\nu
= \frac{N}{|S|}1_S$ satisfies
\begin{multline*}
  \EE[\nu(x,y)\nu(x',y)\nu(x,y')\nu(x',y')
  \nu(x,z-x) \nu(x',z-x') \nu(x,z'-x) \nu(x',z'-x') \\
  \cdot \nu(z-y,y)\nu(z-y',y')\nu(z'-y,y)\nu(z'-y',y') \vert
  x,x',y,y',z,z' \in \ZZ_N] = 1 + o(1)
\end{multline*}
and similar conditions hold if any subset of the twelve $\nu$ factors
in the expectation are erased. Then any corner-free subset of $S$ has size $o(|S|)$. Here a \emph{corner} in $\ZZ_N^2$ is a set of the form
$\{(x,y),(x+d,y),(x,y+d)\}$ for some $d \neq 0$. This claim follows
from Theorem~\ref{thm:rel-mul-sz} by setting $Z = \ZZ_N$, $Z' =
\ZZ_N^2$, $\phi_0(d) = (0,0)$, $\phi_1(d) = (d,0)$, $\phi_2(d) = (0,d)$.

\end{example}

As in \cite [Remark 2.19]{Tao06jam}, we note that the hypothesis
that $\{\phi_i(d) - \phi_j(d): i,j \in J, d \in Z\}$ generate $Z'$ can be dropped
by foliating $Z'$ into cosets. However, this results in a change to
the linear forms hypothesis on $\nu$, namely, that it must be assumed
on every coset.

We shall prove Theorem~\ref{thm:rel-mul-sz} by proving a somewhat more
general removal-type result for arithmetic patterns.

\begin{theorem}
  [Relative arithmetic removal lemma]
  \label{thm:rel-mul-sz-removal}
  For every finite set $J$ and $\e > 0$, there exists $c > 0$ so
  that the following holds. Let $Z, Z', (\phi_j)_{j\in J}, \nu$ be the
  same as in Theorem~\ref{thm:rel-mul-sz}. For any collection of functions
  $\{f_j \colon Z' \to \RR_{\geq 0}\}_{j \in J}$ satisfying $0
  \leq f_j(x) \leq \nu(x)$ for all $x \in Z'$ and $j \in J$, and such that
  \begin{equation} \label{eq:rel-mul-sz-removal-delta}
    \EE\Bigl[ \prod_{j \in J}f_j(a + \phi_j(d)) \Big\vert a \in Z', d
    \in Z \Bigr] \leq c
  \end{equation}
  one can find $A_j \subseteq Z'$ for each $j \in J$ so that
  \begin{equation}
    \label{thm:rel-mul-sz-removal-A0}
    \prod_{j \in J} 1_{A_j} (a + \phi_j(d)) = 0
    \quad\text{for all }a \in Z', d \in Z
  \end{equation}
  and
  \begin{equation}
    \label{thm:rel-mul-sz-removal-eps}
    \EE[f_j(x) 1_{Z' \setminus A_j}(x) \vert x \in Z'] \leq \e
    \quad \text{for all } j \in J.
  \end{equation}
\end{theorem}

Theorem~\ref{thm:rel-mul-sz} follows from
Theorem~\ref{thm:rel-mul-sz-removal} by setting $f_j = f$ for all $j
\in J$ and $\e < \delta/(r+1)$. Indeed, if the conclusion \eqref{eq:rel-mul-sz-small-a} fails,
then Theorem~\ref{thm:rel-mul-sz-removal} implies that there exists $A_j
\subseteq Z'$ for each $j \in J$ satisfying
\eqref{thm:rel-mul-sz-removal-A0} and
\eqref{thm:rel-mul-sz-removal-eps}.
The $A_j$'s cannot have a common intersection, or else
\eqref{thm:rel-mul-sz-removal-A0} fails for $d = 0$. It follows that
$\{Z'\setminus A_j : j \in J\}$ covers $Z'$, and
hence~\eqref{thm:rel-mul-sz-removal-eps} implies that $\EE[f] \leq
\sum_j \EE[f_j 1_{Z'\setminus A_j}] \leq
(r+1)\e < \delta$, which contradicts the hypothesis $\EE[f] \geq \delta$.

\begin{proof}
  [Proof of Theorem~\ref{thm:rel-mul-sz-removal}]
  Let $V = (J,(V_j),r,H)$ be as in the statement of Theorem~\ref{thm:rel-mul-sz}.
  Write $e_j := J \setminus \{j\} \in H$. Define the weighted hypergraph $g$ on $V$ by setting
  \[
  g_{e_j}(x_{e_j}) := f_j(\psi_j(x_{e_j})) \quad\text{for all } j \in J
  \]
  where $\psi_j \colon V_{e_j} \to Z'$ is defined by
  \begin{equation} \label{eq:relsz-pf-psi}
  \psi_j(x_{e_j}) = \sum_{i \in e_j}
(\phi_i(x_i) - \phi_j(x_i))
= a + \phi_j(d)
\end{equation}
where
  \begin{equation} \label{eq:relsz-pf-a-d}
  a = \sum_{i \in J}\phi_i(x_i) \quad\text{and}\quad
  d = - \sum_{i \in J} x_i.
\end{equation}
  Then, for all $x \in V$ and $a,d$ defined in~\eqref{eq:relsz-pf-a-d}, we have
\begin{equation}\label{eq:relsz-pf-g=f}
  \prod_{j \in J} g_{e_j}(x_{e_j}) = \prod_{j \in J} f_j(a + \phi_j(d)).
\end{equation}
The homomorphism $x \mapsto (a,d) \colon V \to Z' \x Z$ given by~\eqref{eq:relsz-pf-a-d} is surjective: the image contains
  $\{(\phi_i(d) - \phi_j(d), 0) : i,j \in J, d \in Z\}$ and hence all
  of $Z' \times \{0\}$. Moreover, the image also contains $\{(-\phi_i(d), d): i \in J, d \in Z\}$. Together, these sets
  generate all of $Z' \times Z$. It follows that $(a,d)$ varies
  uniformly over $Z' \x Z$ as $x$ varies uniformly over $V_J$, and so
  \eqref{eq:relsz-pf-g=f} implies that
  \[
  \EE \Bigl[ \prod_{j \in J} g_{e_j}(x_{e_j}) \Big\vert x \in V_J \Bigr]
  =\EE \Bigl[ \prod_{j \in J} f_j(a+ \phi_j(d)) \Big\vert
  a \in Z', d \in Z \Bigr] \leq c.
  \]
  By the relative hypergraph removal lemma, for $c$ small
  enough (depending on $J$ and $\e$), we can find a subset
  $E'_j
  \subset V_{e_j}$ for each $j \in J$ such that
\begin{equation}\label{eq:relsz-mul-E0}
  \prod_{j \in J} 1_{E'_j}(x_{e_j}) = 0 \quad \text{for all
  } x \in V_J
\end{equation}
and
  \[
  \EE[ g_{e_j}(x_{e_j})
  1_{V_{e_j} \setminus E'_{j}}(x_{e_j}) \vert  x_{e_j} \in
  V_{e_j}] \leq \e/(r+1) \quad \text{for all } j \in J.
  \]
  For each $j \in J$, define $A_j \subseteq Z'$ by
  \begin{equation}
    \label{eq:relsz-mul-Aj}
  A_j := \{ z' \in Z' : \lvert \psi_j^{-1}(z') \cap E'_j \rvert > \tfrac{r}{r+1}
  \lvert \psi_j^{-1}(z') \rvert \}.
\end{equation}
In other words, $A_j$ contains $z' \in Z'$ if the
  hypergraph removal lemma removes less than a $1/(r+1)$
  fraction of the edges in $V_{e_j}$ representing $z'$ via $\psi_j$.

  For any $z' \in Z' \setminus A_j$, on the fiber $\psi^{-1}(z')$ the function $g_{e_j}$ takes the
  common value $f_j(z')$. Furthermore, by \eqref{eq:relsz-mul-Aj}, on this
  fiber, the expectation of $1_{V_{e_j}\setminus E'_j}$ is at least $1/(r+1)$. Hence
  \[
  \EE[f_j(x) 1_{Z' \setminus A_j}(x) \vert x \in Z'] \leq (r+1)
  \EE[g_{e_j}(x_{e_j}) 1_{V_{e_j}\setminus E'_{j}}(x_{e_j}) \vert x_{e_j} \in V_{e_j}]
  \leq \e.
  \]
This proves \eqref{thm:rel-mul-sz-removal-eps}. To prove~\eqref{thm:rel-mul-sz-removal-A0}, suppose for some $a \in
Z', d \in Z$ we have $a + \phi_j(d) \in A_j$ for all $j \in J$. Let
$V^{a,d}_J \subset V_J$ consist of all $x \in V_J$ satisfying
\eqref{eq:relsz-pf-a-d}. Then $\psi_j(x_{e_j}) = a + \phi_j(d)$ for
all $x \in V^{a,d}_J$ by \eqref{eq:relsz-pf-psi}, and in fact $\psi^{-1}_j(a +
\phi_j(d))$ is the projection of $V^{a,d}_J$ onto $V_{e_j}$. By
\eqref{eq:relsz-mul-Aj}, more than an $\frac{r}{r+1}$
fraction of this projection is in $E'_j$. It follows by the pigeonhole principle (or a union bound on the
complement) that there exists some $x \in V_J^{a,d}$ such that $x_{e_j} \in
E'_j$ for every $j \in J$. But this contradicts
\eqref{eq:relsz-mul-E0}. Thus \eqref{thm:rel-mul-sz-removal-A0} holds.
\end{proof}

\section{The relative hypergraph removal lemma} \label{sec:removal-lemma}

\begin{proof}[Proof of Theorem~\ref{thm:rel-removal}]
By Lemma~\ref{lem:nu-upper-regular}, $\nu$ is upper $o(1)$-regular,
so we can apply the weak sparse hypergraph regularity lemma (Theorem~\ref{thm:weak-reg-hyp}) to find functions $\tilde g_e:V_e \rightarrow
[0, 1]$ for every $e \in H$ so that $(g,\tilde g)$ is an
$o(1)$-discrepancy pair. By the counting
lemma (Theorem~\ref{thm:counting-lemma}), we have
  \begin{equation*}
      \EE\Bigl[\prod_{e \in H} \tg_e(x_e)
        \Big\vert
        x \in V_J
        \Bigr]=
        \EE\Bigl[\prod_{e \in H} g_e(x_e)
        \Big\vert
        x \in V_J
        \Bigr]  + o(1) \leq \delta + o(1).
  \end{equation*}
The dense weighted hypergraph removal lemma (Theorem
\ref{thm:dense-hypergraph-removal}) tells us that for each $e \in H$
we can choose $E_e'
\subset V_e$ for which $V_e \setminus E'_e$ has complexity
$O_{\delta}(1)$ (i.e., at most some constant depending on $\delta$)
and such that
\[
\prod_{e \in H} 1_{E'_e}(x_e) = 0 \quad \text{for all $x \in V_J$}
\]
and, as long as $\delta$ is small enough and $N$ is large enough, we have
\begin{equation}
\EE [\tilde g_e(x_e)1_{V_e \setminus E'_e}(x_e)\vert x_e \in V_e] \leq
\e/2\quad \text{for all $e \in H$.} \label{eq:removal-proof-1}
\end{equation}
As $V_e \setminus E'_e$ has complexity $O_{\delta}(1)$, there is a partition of $V_e \setminus E'_e$ into
$O_{\delta}(1)$ hypergraphs $F_{ei}$ each of which is the set of $r$-cliques of some $(r-1)$-uniform hypergraph.
We have
\begin{align}
  \lvert \EE[(\tilde g_e-g_e)(x_e)1_{V_e \setminus
    E'_e}(x_e) \vert x_e \in V_e] \rvert
  &\leq \sum_i \lvert \mathbb{E}[(\tilde g_e-g_e)(x_e)1_{F_{ei}}(x_e) \vert x_e \in
  V_e] \rvert \nonumber
  \\ &\leq \sum_i o(1) =  O_{\delta}(1)o(1) \leq \e/2 \quad \text{for
    all $e \in H$.} \label{eq:removal-proof-2}
\end{align}
We used that $(g_e,\tilde g_e)$ is an $o(1)$-discrepancy pair on each of
the terms of the sum, and the final inequality is true as long as $N$
is large enough. Combining~\eqref{eq:removal-proof-1} and
\eqref{eq:removal-proof-2} we obtain
\[
  \EE[g_e(x_e)1_{V_e \setminus E'_e}(x) \vert x_e \in V_e] \leq \e \quad
  \text{for all $e \in H$}.
  \]
This proves the claim.
\end{proof}

\section{The weak regularity lemma} \label{sec:weak-regularity}


Let $X$ be a finite set and $g:X \rightarrow \mathbb{R}_{\geq 0}$. Let $\mathcal{F}$ be a family of subsets of $X$ which is closed under intersection, $X \in \mathcal{F}$, all subsets of $X$ of size one are in $\mathcal{F}$, and such that, for every $S \in \mathcal{F}$, there is a partition of $X$ which contains $S$ and consists of members of $\mathcal{F}$. For $t \geq 2$, the family $\mathcal{F}$ is {\it $t$-splittable} if for every $S \in \mathcal{F}$ there is a partition $P$ of $X$ into members of $\mathcal{F}$ such that $S \in P$ and $|P| \leq t$. The {\it complexity} $p=p(f)$ of a function $f:X \rightarrow \mathbb{R}_{ \geq 0}$ is the minimum $p$ for which there is a partition $X=S_1 \cup \dots \cup S_p$ into $p$ subsets each in $\mathcal{F}$ such that $f$ is constant on each $S_i$. We call $(g,\tilde g)$ an {\it $\epsilon$-discrepancy pair} if for all $A \in \mathcal{F}$,
$$\big| \mathbb{E}[(g- \tilde g)1_A] \big| \leq \epsilon.$$
All expectations are done with the uniform measure on $X$. For $P$ a partition of $X$, let $g_P$ be the function on $X$ given by $g_P(x)=\frac{\mathbb{E}[g1_A]}{\mathbb{E}[1_A]}$ when $x \in A \in P$. That is, $g_P(x)$ is the conditional expectation of $g(x)$ given the partition $P$ and is constant on any part $A$ of the partition.

The function $g$ we call {\it upper $\eta$-regular} if for every $A \in \mathcal{F}$, we have
$$\mathbb{E}[g1_A]\leq \mathbb{E}[1_A]+\eta.$$
If $g$ is upper $\eta$-regular, $A,B \in \mathcal{F}$, and $\mathcal{F}$ is $t$-splittable, then
\begin{equation}\label{upregcomp}
\mathbb{E}[g1_{B \setminus A}]\leq \mathbb{E}[1_{B \setminus A}]+(t-1)\eta.
\end{equation}
Indeed, in this case $B \setminus A$ can be partitioned into $t-1$ sets in $\mathcal{F}$ (we first split with respect to $A$ and then consider the intersections of the parts of the partition with $B$). Applying the upper $\eta$-regularity condition to each of these sets and summing up the inequalities, we arrive at (\ref{upregcomp}).

Following Scott~\cite{Sco11}, let $\phi:\mathbb{R}_{\geq 0} \rightarrow \mathbb{R}_{\geq 0}$ be the convex function given by
\begin{equation*}
\phi(u)=
\begin{cases} u^2 & \text{if $u \leq 2$,}
\\
4u-4 &\text{otherwise.}
\end{cases}
\end{equation*}

\noindent
For a partition $P$ of $X$, let $\phi(P)=\mathbb{E}[\phi\left(g_P\right)]$, which is the mean $\phi$-density of $g$ with respect to the partition $P$. As $\phi$ takes only nonnegative values and $\phi(u) \leq 4u$,  we have $$0 \leq \phi(P) \leq 4\mathbb{E}[g_P] = 4\mathbb{E}[g].$$
Also, by the convexity of $\phi$, it follows that if $P'$ is a refinement of $P$, then $\phi(P') \geq \phi(P)$.

\begin{lemma} \label{lem:weak-reg-general}
Let $X$ and $\mathcal{F}$ as above be such that $\mathcal{F}$ is $t$-splittable. Let $0<\epsilon, \eta < 1$ and $T=t^{20/\epsilon^2}$. For any $g:X\rightarrow \mathbb{R}_{\geq 0}$ which is upper $\eta$-regular with $\eta \leq \frac{\epsilon}{8tT}$, there is $\tilde g:X \rightarrow [0,1]$ with complexity at most $T$ such that $(g,\tilde g)$ is an $\epsilon$-discrepancy pair.
\end{lemma}

\begin{proof}
Let $\alpha=\frac{\epsilon^2}{4}$. We first find a partition $P$ of $X$ into members of $\mathcal{F}$ with $|P| \leq t^{5/\alpha}=T$ such that for any refinement $P'$ of $P$ into members of $\mathcal{F}$ with $|P'| \leq t|P|$, we have
$\phi(P')-\phi(P) < \alpha$. In order to construct $P$, we first recursively construct a sequence $P_0,P_1,\ldots$ of finer partitions of $X$ into members of $\mathcal{F}$ so that $|P_j| \leq t^j$ and $\phi(P_j)\geq j\alpha$. We begin by considering the trivial partition $P_0=\{X\}$, which satisfies $\phi(P_0) \geq 0$. At the beginning of step $j+1$, we have a partition $P_j$ of $X$ into members of $\mathcal{F}$ with $|P_j| \leq t^j$ and $\phi(P_j) \geq j\alpha$. If there exists a refinement $P_{j+1}$ of $X$ into members of $\mathcal{F}$  with $|P_{j+1}| \leq t|P_j|$ and $\phi(P_{j+1}) \geq \phi(P_j)+\alpha$, then we continue to step $j+2$. Otherwise, we may pick $P=P_j$ to be the desired partition. Note that this process must stop after at most $5/\alpha$ steps since $5 > 4(1+\eta) \geq 4\mathbb{E}[g] \geq \phi(P_j) \geq j\alpha$, where the second inequality follows from $g$ being upper $\eta$-regular. We therefore arrive at the desired partition $P$.

Let $P:X=S_1 \cup \dots \cup S_p$. Let $\tilde g:X \rightarrow [0,1]$, where $\tilde g = g_P \wedge 1$ is  the minimum of $g_P$ and the constant function $1$. We will show that $(g_P,\tilde g)$ is an $\frac{\epsilon}{4}$-discrepancy pair and $(g_P,g)$ is a $\frac{3\epsilon}{4}$-discrepancy pair, which implies by the triangle inequality that $(g,\tilde g)$ is an $\epsilon$-discrepancy pair.  As $\tilde g$ has complexity at most $|P| \leq T$, this will complete the proof.

We first show $(g_P,\tilde g)$ is an $\frac{\epsilon}{4}$-discrepancy pair. Note that $g_P-\tilde g$ is nonnegative and constant on each part of $P$. If $S_i \in P$ and $g_P-\tilde g>0$ on $S_i$, then also $g_P>1$ and $\tilde g=1$ on $S_i$. As $g$ is upper $\eta$-regular, we have $\mathbb{E}[g1_{S_i}] \leq \mathbb{E}[1_{S_i}]+\eta$ and hence $\mathbb{E}[(g-\tilde g)1_{S_i}] \leq \eta$. Therefore, by summing over all parts in the partition $P$, we see that if $A \in \mathcal{F}$,
$$0 \leq \mathbb{E}[(g_P-\tilde g)1_A] \leq \mathbb{E}[(g_P-\tilde g)] \leq  \eta |P| \leq \eta T \leq \frac{\epsilon}{4},$$
and $(g_P,\tilde g)$ is an $\frac{\epsilon}{4}$-discrepancy pair.

We next show that $(g_P,g)$ is a $\frac{3\epsilon}{4}$-discrepancy pair, which completes the proof. Suppose for contradiction that there is $A \in \mathcal{F}$ such that $$|\mathbb{E}[(g_P-g)1_A]| > \frac{3\epsilon}{4}.$$
Let $B$ be the union of all $S_i \cap A$, where $S_i \in P$, for which both $\mathbb{E}[1_{S_i \cap A}] \geq t\eta$ and $\mathbb{E}[1_{S_i \setminus A}] \geq t\eta$.

We claim that for each $S_i \in P$, we have
\begin{equation}\label{upbd2}|\mathbb{E}[(g_P-g)(1_{A \cap S_i}-1_{B \cap S_i})]| \leq 2t\eta.\end{equation}
Indeed, if $B \cap S_i=A \cap S_i$, then the left hand side of (\ref{upbd2}) is $0$. Otherwise, $\mathbb{E}[1_{A \cap S_i}] \leq t\eta$ or $\mathbb{E}[1_{S_i \setminus A}] \leq t\eta$. In the first case, when $\mathbb{E}[1_{A \cap S_i}] \leq t\eta$, we have $1_{B \cap S_i}$ is identically $0$, as well as  $$\mathbb{E}[g1_{A \cap S_i}] \leq  \mathbb{E}[1_{A \cap S_i}]+\eta \leq (t+1)\eta$$
and $$\mathbb{E}[g_P1_{A \cap S_i}]=\frac{\mathbb{E}[g1_{S_i}]}{\mathbb{E}[1_{S_i}]}\mathbb{E}[1_{A \cap S_i}] \leq \frac{\left(\mathbb{E}[1_{S_i}]+\eta\right)}{\mathbb{E}[1_{S_i}]}\mathbb{E}[1_{A \cap S_i}] \leq \mathbb{E}[1_{A \cap S_i}]+\eta \leq (t+1)\eta,$$
from which (\ref{upbd2}) follows. In the second case, when  $\mathbb{E}[1_{S_i \setminus A}] \leq t\eta$, we again have $1_{B \cap S_i}$ is identically $0$, so that
\begin{align*}\mathbb{E}[(g-g_P)(1_{A \cap S_i}-1_{B \cap S_i})] &=  \mathbb{E}[(g-g_P)1_{A \cap S_i}]=
\mathbb{E}[(g-g_P)(1_{S_i}-1_{S_i \setminus A})]
\\ &= \mathbb{E}[(g-g_P)1_{S_i}]-\mathbb{E}[(g-g_P)1_{S_i \setminus A}] = -\mathbb{E}[(g-g_P)1_{S_i \setminus A}] ,
\end{align*}
and similar to the first case, using (\ref{upregcomp}) to estimate $\mathbb{E}[g1_{S_i \setminus A}]$ and  $\mathbb{E}[g_P1_{S_i \setminus A}]$,  we get (\ref{upbd2}).

Notice that
$$|\mathbb{E}[(g_P-g)1_A]-\mathbb{E}[(g_P-g)1_B]|=|\mathbb{E}[(g_P-g)(1_A-1_B)]| \leq |P|2t\eta \leq \frac{\epsilon}{4},$$
where the first inequality follows by using (\ref{upbd2}) for each part $S_i$ and the triangle inequality. Hence, $$|\mathbb{E}[(g_P-g)1_B]| \geq |\mathbb{E}[(g_P-g)1_A]|- |\mathbb{E}[(g_P-g)1_A]-\mathbb{E}[(g_P-g)1_B]| > \frac{3\epsilon}{4}-\frac{\epsilon}{4} = \frac{\epsilon}{2}.$$

Let $\hat P$ be the refinement of $P$ where $S_i$ is also in $\hat P$ if $B \cap S_i= \emptyset$ and otherwise $S_i \cap B$ and $S_i \setminus B$ are parts of $\hat P$, and let $P'$ be a refinement of $\hat P$ into at most $t|P|$ members of $\mathcal{F}$. The refinement $P'$ exists as $\mathcal{F}$ is $t$-splittable and is closed under intersections, $P$ consists of members of $\mathcal{F}$, $A \in \mathcal{F}$, and $S_i \cap B= S_i \cap A \in \mathcal{F}$ if $S_i \cap B \in \hat P$. As $P'$ is a refinement of $\hat P$ which is a refinement of $P$, we have $\phi(P') \geq \phi(\hat P) \geq \phi(P)$. Let $R \in \{S_i,S_i \cap B,S_i \setminus B\}$, where $S_i$ is a part of $P$ that is refined into two parts in  $\hat P$, so that $\mathbb{E}[1_{R}] \geq t\eta$.  Letting $u=\frac{\mathbb{E}[g1_R]}{\mathbb{E}[1_R]}$, we see, since $g$ is upper $\eta$-regular and using (\ref{upregcomp}), that $u \leq 1+t\eta (t\eta)^{-1} = 2$ and hence $\phi(u)=u^2$. It follows, by considering the functions pointwise, that  $\phi(g_{\hat P})-\phi(g_P)=g_{\hat P}^2-g_P^2$. Hence,
\begin{align*}
  \phi(P')-\phi(P) &\geq
  \phi(\hat P)-\phi(P)=\mathbb{E}[g_{\hat P}^2]-
\mathbb{E}[g_P^2]=\mathbb{E}[g_{\hat P}^2-g_P^2]=
\mathbb{E}[\left(g_{\hat P}-g_P\right)^2] \\ &\geq
\mathbb{E}[(g_{\hat P} - g_P)1_B]^2 =
\mathbb{E}[(g - g_P)1_B]^2 > \frac{\epsilon^2}{4}=\alpha.
\end{align*}
The third equality above is the Pythagorean identity, which uses that $\hat P$ is a refinement of $P$, and the second inequality is an application of the Cauchy-Schwarz inequality. However, since $P'$ is a refinement of $P$ consisting of members of $\mathcal{F}$ with $|P'| \leq t|P|$, this contradicts $\phi(P')-\phi(P) < \alpha$ from the definition of $P$ and completes the proof.
\end{proof}

To establish the weak hypergraph regularity lemma, Theorem \ref{thm:weak-reg-hyp}, we use Lemma \ref{lem:weak-reg-general} with $X=V_1 \times \cdots \times V_r$  and $\mathcal{F}$ being the family of subsets of $X$ which form the $r$-cliques of some $r$-partite $(r-1)$-uniform hypergraph with parts $V_1,\ldots,V_r$. Noting that $\mathcal{F}$ is $2^r$-splittable in this case, we obtain Theorem \ref{thm:weak-reg-hyp}.

\section{The counting lemma} \label{sec:counting-lemma}

The three main ingredients in our proof of the counting
lemma (Theorem~\ref{thm:counting-lemma}) are as follows.
\begin{enumerate}
\item A standard telescoping argument~\cite{BCLSV08} in the dense case, i.e., when $\nu = 1$.
\item Repeated applications of the Cauchy-Schwarz
  inequality. This is a standard technique in this area, e.g., \cite{Gow01,Gow07,GT08,Tao06jam}.
\item \emph{Densification}. This is the main new ingredient in our proof. At
  each step, we reduce the problem of counting $H$ in a particular weighted hypergraph
  to that of counting $H$ in a modified weighted hypergraph.
  For an edge $e \in H$, we replace the triple $(\nu_e, g_e, \tg_e)$ by a new triple
   $(1, g'_e, \tg'_e)$ with $0 \leq
  g'_e, \tg'_e \leq 1$ and such that $(g'_e,\tg'_e)$ is an $\e'$-discrepancy
  pair for some $\e' = o_{\e \to 0}(1)$. By repeatedly applying
  this reduction to all $e \in H$ (we use induction), we
  reduce the counting lemma to the dense case.
\end{enumerate}

We developed the densification technique in our earlier
paper~\cite{CFZ14}, where we proved a sparse counting lemma in graphs. We
have significantly simplified a number of technical steps
from~\cite{CFZ14} in order to extend the densification technique to
hypergraphs here.

\subsection{Telescoping argument} The following argument allows us to
prove the counting lemma in the dense case, i.e., when $0 \leq g \leq 1$.

\begin{lemma}[Telescoping discrepancy argument for dense hypergraphs]
  \label{lem:telescoping}
  Theorem~\ref{thm:counting-lemma} holds if we assume that there
  is some $e_1 \in H$ so that $\nu_e
  = 1$ for all $e \in H \setminus \{e_1\}$. In fact, in this case,
\begin{equation}\label{eq:telescoping}
\biggl\lvert
\EE\Bigl[\prod_{e \in H} g_e(x_e)
        \Big\vert
        x \in V_J
        \Bigr]
        -
        \EE\Bigl[\prod_{e \in H} \tg_e(x_e)
        \Big\vert
        x \in V_J
        \Bigr]
        \biggr\rvert \leq \abs{H} \e.
      \end{equation}
    \end{lemma}

Lemma~\ref{lem:telescoping} uses only the assumption that $(g_e,
\tg_e)$ is an $\e$-discrepancy pair for every $e \in H$ and nothing about the linear forms
condition on $\nu$. Recall that for each fixed $e \in H$, the condition that
$(g_e, \tg_e)$ is an $\e$-discrepancy pair means that
for all subsets $B_f \subseteq V_f$, $f \in \del e$, we have
  \begin{equation}
    \label{eq:disc-pair-restated}
    \Bigl\lvert
    \EE\Bigl[
    (g_e(x_e) - \tg_e(x_e)) \prod_{f \in \del e}
        1_{B_f}(x_f)
        \Big\vert x_e \in V_e
        \Bigr]\Bigr\rvert \leq \e.
  \end{equation}
This is equivalent to the condition that for all functions $u_f \colon
V_f \to [0,1]$, $f \in \del e$, we have
  \begin{equation}
    \label{eq:disc-pair-u}
    \Bigl\lvert
    \EE\Bigl[
    (g_e(x_e) - \tg_e(x_e)) \prod_{f \in \del e}
        u_f(x_f)
        \Big\vert x_e \in V_e
        \Bigr]\Bigr\rvert \leq \e.
  \end{equation}
Indeed, the expectation is linear in each $u_f$ and hence the extrema
occur when the $u_f$'s are $\{0,1\}$-valued, thereby
reducing to \eqref{eq:disc-pair-restated}.

\begin{proof} Let $h = \abs{H}$ and order the edges of $H \setminus \{e_1\}$
  arbitrarily as $e_2, \dots, e_h$. We can write the left-hand side of
  \eqref{eq:telescoping}, without the absolute values, as a telescoping sum
  \begin{equation}\label{eq:telescope-sum}
  \sum_{t=1}^h \EE\Bigl[\Bigl(\prod_{s=1}^{t-1}
  \tg_{e_s}(x_{e_s})\Bigr) (g_{e_t}(x_{e_t}) -
  \tg_{e_t}(x_{e_t})) \Bigl(\prod_{s=t+1}^{h} g_{e_s}(x_{e_s})\Bigr)
  \Big\vert x \in V_J\Bigr].
\end{equation}
For the $t$-th term in the sum, when we fix the value of $x_{J \setminus e_t} \in
  V_{J \setminus e_t}$, the expectation has the form
  \begin{equation}
    \label{eq:telescope-term-u}
    \EE\Bigl[(g_{e_t}(x_{e_t}) - \tg_{e_t}(x_{e_t})) \prod_{f \in \del e_t}
        u_f(x_f) \Big\vert x_{e_t} \in V_{e_t}\Bigr]
    \end{equation}
    for some functions $u_f \colon V_f \to [0,1]$ (here we used the
    key fact that $g_{e_s} \leq 1$ for all $s > 1$ and $\tg_{e_s} \leq
    1$ for all $s$). Since $(g_{e_t}, \tg_{e_t})$ is an
    $\e$-discrepancy pair, \eqref{eq:disc-pair-u} implies that \eqref{eq:telescope-term-u} is bounded in
    absolute value by $\e$.  The same bound holds after we vary $x_{J
      \setminus e_t} \in V_{J \setminus e_t}$. So every term in
    \eqref{eq:telescope-sum} is bounded by $\e$ in absolute value, and
    hence \eqref{eq:telescope-sum} is at most $h\e$ in absolute value.
\end{proof}

\subsection{Strong linear forms}
\label{subsect:stronglinearform}
The main result of this subsection tells us that $\nu$ can be replaced by the constant function $1$ in counting expressions.
Though somewhat technical in detail, the main idea of the proof is quite simple and may be summarized as
follows: we use the Cauchy-Schwarz inequality to double each vertex $j$ of a certain edge in turn, at each step
majorizing those edges which do not contain $j$.
This method is quite standard in the field. In the work of Green and Tao, it is used to prove
generalized von Neumann theorems
\cite[Prop.~5.3]{GT08}, \cite[Thm.~3.8]{Tao06jam}, although the
statement of our lemma is perhaps more similar to the uniform
distribution property \cite[Prop.~6.2]{GT08},
\cite[Prop.~5.1]{Tao06jam}.

We begin by using a similar method to prove a somewhat easier result. It shows that if $\nu$ satisfies the $H$-linear forms condition then $(\nu, 1)$ is an $o(1)$-discrepancy pair, which implies Lemma~\ref{lem:nu-upper-regular}.


\begin{lemma} \label{lem:nu-1-disc}
Let $e$ be a finite set, $V_j$ a finite set for each $j \in
e$, and $V_e = \prod_{j\in e} V_j$. Then, for any function $\nu \colon V_e \to
\RR$ and any collection of $B_f \subseteq V_f$ for $f \in \del e$,
\begin{equation}
  \label{eq:nu-1-disc}
  \Bigl\lvert \EE\Bigl[ (\nu_e(x_e) - 1) \prod_{f \in \del e} 1_{B_f}(x_f)
      \Big\vert
      x_e \in V_e \Bigr] \Bigr\rvert
      \leq
      \EE \Bigl[ \prod_{\omega \in \{0,1\}^e}(\nu_e(x_e^{(\omega)}) -
      1) \Big\vert x_e^{(0)}, x_e^{(1)} \in V_e \Bigr]^{1/2^{\abs{e}}}.
\end{equation}
\end{lemma}

Lemma~\ref{lem:nu-1-disc} follows from a direct application of the
Gowers-Cauchy-Schwarz~\cite{Gow01} inequality for hypergraphs (see~\cite{CFZab}). We include
the proof here for completeness.

\begin{proof}
  For $\emptyset \subseteq d \subseteq e$, let
  \[
    X_d := \prod_{\omega \in \{0,1\}^d}(v_e(x_{e\setminus d},
    x_d^{(\omega)}) - 1), \qquad
    Y_d := \prod_{\substack{f \in \del e \\ f \supseteq d}} \prod_{\omega \in \{0,1\}^d}
    1_{B_f}(x_{f\setminus d}, x_d^{(\omega)}),
\]
  and
  \[
  Q_d := \EE[ X_d Y_d \vert x_{e \setminus d} \in V_{e \setminus d}, \
  x_d^{(0)}, x_d^{(1)} \in V_d ].
  \]
  Then \eqref{eq:nu-1-disc} can be written as $\abs{Q_\emptyset} \leq
  Q_e^{1/2^{\abs{e}}}$. By induction, it suffices to show that $Q_d^2 \leq Q_{d \cup
    \{j\}}$ whenever $j \in e \setminus d$. Let $Y_d = Y_d^{\ni j} Y_d^{\not\ni
    j}$ where $Y_d^{\ni j}$ consists of all the factors in $Y_d$ that
  contain $x_j$ in the argument, and $Y_d^{\not\ni j}$ consists of all
  other factors. By the Cauchy-Schwarz inequality, we have
  \[
  Q_d^2
  = \EE[\EE[X_d Y_d^{\ni j} \vert x_j \in V_j] Y_d^{\not\ni j}]^2
  \leq \EE[ \EE[X_d Y_d^{\ni j} \vert x_j \in V_j]^2] \EE[(Y_d^{\not
    \ni j})^2]
  \leq Q_{d\cup \{j\}},
  \]
  since $Q_{d\cup\{j\}} = \EE[ \EE[X_d Y_d^{\ni j} \vert x_j \in
  V_j]^2]$ and $0 \leq Y_d^{\not \ni j} \leq 1$, where the outer
  expectations are taken over all free variables. This shows that $Q_d^2
  \leq Q_{d \cup \{j\}}$. Hence, $\abs{Q_\emptyset} \leq
  Q_e^{1/2^{\abs{e}}}$, as desired.
\end{proof}

The next lemma is very similar, except that now we need to
invoke the linear forms condition.

\begin{lemma}[Strong linear forms]  \label{lem:strong-linear-forms}
  Let $V = (J, (V_j)_{j \in J}, r, H)$ be a hypergraph system and let
  $\nu$ be a weighted hypergraph on $V$ satisfying the linear forms
  condition. Let $e_1 \in H$. For each $\iota \in \{0,1\}$ and $e \in
  H \setminus \{e_1\}$, let $g^{(\iota)}_e \colon V_e \to \RR_{\geq 0}$ be a
  function so that either $g_e^{(\iota)} \leq 1$ or
  $g_e^{(\iota)}\leq \nu_e$ holds. Then
  \begin{equation}
    \label{eq:slf}
    \EE\Bigl[
    (\nu_{e_1}(x_{e_1}) - 1)
    \prod_{\iota \in \{0,1\}} \Bigl( \prod_{e \in H \setminus \{e_1\}}
    g_e^{(\iota)}(x^{(\iota)}_e)
    \Bigr)
    \Big\vert
    x_J^{(0)},x_J^{(1)} \in V_J; x_{e_1}^{(0)} =
    x_{e_1}^{(1)} = x_{e_1}
    \Bigr] = o(1).
  \end{equation}
\end{lemma}

In \eqref{eq:slf} the notation $x_{e_1}^{(0)} =
    x_{e_1}^{(1)} = x_{e_1}$ means that $x_j^{(0)}, x_j^{(1)}, x_j$
    are taken to be the same for all $j \in e_1$.
Recall that we write $o(1)$ for a quantity that tends to zero as $N
\to \infty$.

\begin{proof}
  For each $\iota \in \{0,1\}$ and $e \in H\setminus\{e_1\}$, let $\og_e^{(\iota)}$ be either $1$ or $\nu_e$ so that $g_e^{(\iota)} \leq
  \og_e^{(\iota)}$ holds. For $\emptyset \subseteq d \subseteq e_1$, define
  \begin{align*}
    X_d &:= \prod_{\omega \in \{0,1\}^d}(\nu_{e_1}(x_{e_1\setminus d},
    x_d^{(\omega)}) - 1), \\
    Y_d &:= \prod_{\iota \in \{0,1\}} \prod_{e \in H \setminus \{e_1\}}
    \prod_{\omega \in \{0,1\}^{e \cap d}}
      \left\{
        \begin{array}{ll}
          g_e^{(\iota)}(x_{e\setminus e_1}^{(\iota)}, x_{d}^{(\omega)}, x_{e \cap e_1 \setminus d}) & \text{if } e
          \supseteq d \\
          \og_e^{(\iota)}(x_{e\setminus e_1}^{(\iota)}, x_{e\cap
            d}^{(\omega)}, x_{e \cap e_1 \setminus d}) & \text{if } e
          \nsupseteq d \\
          \end{array}
          \right\},
        \end{align*}
and
\[
        Q_d := \EE \bigl[ X_dY_d \big\vert x^{(0)}_{(J\setminus
    e_1) \cup d},x^{(1)}_{(J\setminus e_1)\cup d} \in V_{(J
    \setminus e_1) \cup d}, \ x_{e_1 \setminus d} \in V_{e_1 \setminus
    d} \bigr].
\]
We observe that $Q_\emptyset$ is equal to the left-hand side of
  \eqref{eq:slf} and
  \[
  Q_{e_1} = \EE \Bigl[ \prod_{\omega \in \{0,1\}^{e_1}} (
  \nu_{e_1}(x_{e_1}^{(\omega)}) - 1) \
  \prod_{\iota \in \{0,1\}} \prod_{e \in H \setminus \{e_1\}}
  \prod_{\omega \in \{0,1\}^{e \cap e_1}}
  \og_e^{(\iota)}(x_{e\setminus e_1}^{(\iota)}, x_{e\cap
    e_1}^{(\omega)})
  \Big\vert
  x_J^{(0)},x_J^{(1)} \in V_J
  \Bigr] = o(1)
  \]
  by the linear forms condition~\eqref{eq:H-lfc}.\footnote{This is where the weak $2$-blow-up of $H$ arises, since the estimate $Q_{e_1} = o(1)$ only relies upon knowing that $\nu$ has roughly the expected density for certain subgraphs of the weak $2$-blow-up.} Indeed, after we expand $\prod_{\omega \in \{0,1\}^{e_1}} (
  \nu_{e_1}(x_{e_1}^{(\omega)}) - 1)$, every term in $Q_{e_1}$ has the
  form of
  \eqref{eq:H-lfc} (since $\og_e^{(\iota)}$ is $1$ or $\nu_e$). Thus
  $Q_{e_1}$ is the sum of $2^{|e_1|}$ terms, each of which is $\pm(1 +
  o(1))$ by the linear forms condition, and they cancel accordingly to $o(1)$.

  We claim that if $j \in e_1 \setminus d$ then
  \begin{equation}
    \label{eq:slf-cs}
    \lvert Q_d \rvert \leq (1 + o(1)) Q_{d \cup
      \{j\}}^{1/2},
  \end{equation}
  from which it would follow by induction that
  \[
  \abs{\text{LHS of \eqref{eq:slf}}} = \abs{Q_\emptyset} \leq (1 +
  o(1))Q_{e_1}^{1/2^r} = o(1).
  \]
  Now we prove~\eqref{eq:slf-cs}. Let $Y_d = Y_d^{\ni j} Y_d^{\not\ni
    j}$ where $Y_d^{\ni j}$ consists of all the factors in $Y_d$ that
  contain $x_j$ in the argument, and $Y_d^{\not\ni j}$ consists of all
  other factors. Using the Cauchy-Schwarz inequality and $Y_d^{\not\ni j} \leq \ol Y_d^{\not\ni
    j}$ one has
  \begin{align}
    Q_d^2 &=
    \EE [ \EE[ X_d Y_d^{\ni j} \vert x_j \in V_j]
    Y_d^{\not\ni j} ]^2
\leq
    \EE [\EE[ X_d Y_d^{\ni j} \vert x_j \in V_j]^2
    Y_d^{\not\ni j} ] \
    \EE [Y_d^{\not\ni j} ] \nonumber
    \\
    &\leq
    \EE [\EE[ X_d Y_d^{\ni j} \vert x_j \in V_j]^2
    \ol Y_d^{\not\ni j}] \ \EE [\ol Y_d^{\not\ni j}]
= Q_{d\cup\{j\}} \ \EE [\ol Y_d^{\not\ni j}] \label{eq:slf-post-cs}
  \end{align}
  where the outer expectations are taken over all free variables. The second factor in \eqref{eq:slf-post-cs} is $1
  + o(1)$ by the linear forms condition~\eqref{eq:H-lfc} as $\ol
  Y_d^{\not\ni j}$ consists only of $\nu$ factors. This proves \eqref{eq:slf-cs}.
\end{proof}

\subsection{Counting lemma proof}

As already mentioned, the main idea of the following proof is a process called densification, where we reduce the problem of counting $H$ in a sparse hypergraph to that of counting $H$ in a dense hypergraph by replacing sparse edges with dense edges one at a time. Several steps are needed to densify a given edge $e_1$. The first step is to double all vertices outside of $e_1$ and to majorize $g_{e_1}$ by $\nu_{e_1}$. We then use the strong linear forms condition to remove the edge corresponding to $e_1$ entirely. This leaves us with the seemingly harder problem of counting the graph $H'$ consisting of two copies of $H\char92\{e_1\}$ joined along the vertices of $e_1$. However, an inductive hypothesis tells us that we can count copies of $H \char92\{e_1\}$. The core of the proof is in showing that this allows us to replace one of the copies of $H\char92\{e_1\}$ in $H'$ by a dense edge, thus reducing our problem to that of counting $H$ with one edge replaced by a dense edge.

\begin{proof}[Proof of Theorem~\ref{thm:counting-lemma}]
  We use induction on $\abs{\set{e \in H : \nu_e \neq 1}}$. When
  $\abs{\set{e \in H : \nu_e \neq 1}} = 0$ or $1$, the result follows
  from Lemma~\ref{lem:telescoping}. Now take $e_1 \in H$ so that
  $\nu_{e_1} \neq 1$.

  We assume that $\abs{J}$ is a fixed constant. We write $o(1)$ for a
  quantity that tends to zero as $N \to \infty$ and $o_{\e \to 0}(1)$
  for a quantity that tends to zero as $N \to \infty$ and $\e \to 0$.
  We need to show that the following quantity is $o_{\e \to 0}(1)$:
  \begin{multline} \label{eq:counting-initial-split}
        \EE\Bigl[\prod_{e \in H} g_e(x_e)
        \Big\vert
        x \in V_J
        \Bigr] -
        \EE\Bigl[\prod_{e \in H} \tg_e(x_e)
        \Big\vert
        x \in V_J
        \Bigr] \\
        =
    \EE\Bigl[ g_{e_1}(x_{e_1}) \Bigl( \!\!\! \prod_{e \in H \setminus \{e_1\}} \!\!\!
        g_e(x_e) - \!\!\!\prod_{e \in H \setminus \{e_1\}} \!\!\!
        \tg_e(x_e)\Bigr) \Big\vert x
        \in V_J \Bigr]
    +
    \EE\Bigl[(g_{e_1}(x_{e_1}) - \tg_{e_1}(x_{e_1})) \Bigl( \!\!\!\prod_{e
      \in H \setminus \{e_1\}} \!\!\!
                \tg_e(x_e) \Bigr) \Big\vert
              x \in V_J\Bigr].
  \end{multline}
  The second term in~\eqref{eq:counting-initial-split} is at most $\e$ in absolute
  value since $(g_{e_1},\tg_{e_1})$ is an $\e$-discrepancy pair and $\tg \leq
  1$ (e.g., see proof of
  Lemma~\ref{lem:telescoping}). It remains to show that the first
  term in \eqref{eq:counting-initial-split} is $o_{\e \to
    0}(1)$.

    Define functions $\nu'_{e_1}, g'_{e_1}, \tg'_{e_1} \colon V_{e_1}
    \to \RR_{\geq 0}$ by
  \begin{align}
    \nu'_{e_1}(x_{e_1}) &:= \EE \Bigl[\prod_{e \in H \setminus \{e_1\}}
      \nu_e(x_e) \Big\vert x_{J\setminus e_1} \in V_{J\setminus e_1}\Bigr] \label{eq:nu'},
      \\
     g'_{e_1}(x_{e_1}) &:= \EE \Bigl[\prod_{e \in H \setminus \{e_1\}}
      g_e(x_e)
      \Big\vert  x_{J\setminus e_1} \in V_{J\setminus e_1} \Bigr] \label{eq:g'},
    \\
     \tg'_{e_1}(x_{e_1}) &:= \EE \Bigl[\prod_{e \in H \setminus \{e_1\}}
      \tg_e(x_e)
      \Big\vert x_{J\setminus e_1} \in V_{J\setminus e_1} \Bigr] \label{eq:tg'}.
  \end{align}
 We have $g'_{e_1} \leq \nu'_{e_1}$ and $\tg_{e_1} \leq 1$
 (pointwise). In the rest of this proof, unless otherwise specified,
 expectations are for functions on $V_{e_1}$ with arguments varying uniformly
 over $V_{e_1}$. The linear forms condition~\eqref{eq:H-lfc}
 implies that $\EE[\nu'_{e_1}]= 1 + o(1)$ and $\EE[(\nu'_{e_1})^2] = 1 +
 o(1)$, so that\footnote{In fact, the only assumptions on $\nu$ needed for the
   proof of Theorem~\ref{thm:counting-lemma} are \eqref{eq:lfc-nu'} and the strong
   linear forms condition, Lemma~\ref{lem:strong-linear-forms}, as
   well as analogous conditions for other choices of $e_1 \in H$ and allowing some
   subset of the functions $\nu_e$ to be replaced by $1$.\label{ft:conditions}}
 \begin{equation} \label{eq:lfc-nu'}
  \EE[(\nu'_{e_1} - 1)^2] = o(1).
\end{equation}
The square of the first term in \eqref{eq:counting-initial-split}
  equals
  \begin{align}
    \EE[g_{e_1} (g'_{e_1} - \tg'_{e_1})]^2 \leq \EE[g_{e_1} (g'_{e_1}
    - \tg'_{e_1})^2] \ \EE[g_{e_1}] &\leq \EE[\nu_{e_1} (g'_{e_1} -
    \tg'_{e_1})^2 ] \ \EE [\nu_{e_1}] \nonumber
    \\
    &= (\EE[(g'_{e_1} - \tg'_{e_1})^2] + o(1))(1 +
    o(1)) \label{eq:count-main-cs}.
 \end{align}
 The first inequality above is due to the Cauchy-Schwarz
 inequality. In the final step, both factors are estimated using
 Lemma~\ref{lem:strong-linear-forms} (for the first factor, expand the square
 $(g'_{e_1} - \tg_{e_1}')^2$ and apply Lemma~\ref{lem:strong-linear-forms} term by term). Continuing
 \eqref{eq:count-main-cs} it suffices to show that the following
 quantity is $o_{\e \to 0}(1)$:
 \begin{equation} \label{eq:cutoff}
\EE[(g'_{e_1} - \tg'_{e_1})^2]
= \EE[(g'_{e_1} - \tg'_{e_1})(g'_{e_1} - g'_{e_1} \wedge 1)] +
\EE[(g'_{e_1} - \tg'_{e_1})(g'_{e_1} \wedge 1 - \tg'_{e_1})]
  \end{equation}
  (here $a \wedge b := \min\{a,b\}$). That is, we are capping the
  weighted hypergraph $g'_{e_1}$ by 1. Since $\nu'_{e_1}$ is very close to 1 by \eqref{eq:lfc-nu'},
  this should not result in a large loss. Indeed,
  since $0 \leq g'_{e_1} \leq \nu'_{e_1}$, we have
  \begin{equation} \label{eq:over-control}
    0 \leq g'_{e_1} - g'_{e_1} \wedge 1 = \max\{g'_{e_1} - 1,0\} \leq
    \max\{\nu'_{e_1} - 1,0\} \leq \lvert \nu'_{e_1} - 1\rvert.
  \end{equation}
  Using~\eqref{eq:over-control}, $g'_{e_1} \leq \nu'_{e_1}$, and
  $\tg'_{e_1} \leq 1$, we bound the magnitude of the first term on the
  right-hand side of \eqref{eq:cutoff} by
  \[
  \EE[(\nu'_{e_1} + 1)\abs{\nu'_{e_1} - 1}]
  = \EE[(\nu'_{e_1} - 1)\abs{\nu'_{e_1} - 1}] + 2 \EE[\abs{\nu'_{e_1} - 1}]
  \leq \EE[(\nu'_{e_1} - 1)^2] + 2 \EE[(\nu'_{e_1}-1)^2]^{1/2}
  = o(1)
  \]
  by the triangle inequality, the Cauchy-Schwarz inequality, and \eqref{eq:lfc-nu'}. To estimate the second term
  on the right-hand side of~\eqref{eq:cutoff}, we need the following claim.

  \begin{claim}$(g'_{e_1}\wedge 1,
  \tg'_{e_1})$ is an $\e'$-discrepancy pair with $\e' = o_{\e\to
    0}(1)$.
  \end{claim}

\noindent
  \emph{Proof of Claim.}
  We need to show that, whenever $B_{f} \subseteq V_f$ for
  all $f \in \del e_1$, we have
  \begin{equation} \label{eq:cap-disc}
    \EE\Bigl[
    (g'_{e_1}(x_{e_1}) \wedge 1 - \tg'_{e_1}(x_{e_1})) \prod_{f \in \del e_1}
        1_{B_f}(x_f)
        \Big\vert x_{e_1} \in V_{e_1}
        \Bigr] = o_{\e \to 0}(1).
    \end{equation}
    Define $g''_{e_1} \colon V_{e_1} \to \RR_{\geq 0}$ by
  $g''_{e_1}(x_{e_1}) = \prod_{f \in \del e_1} 1_{B_f}(x_f)$. So the
  left-hand side of \eqref{eq:cap-disc} is equal to
  \begin{equation}        \label{eq:disc-cap-split}
    \EE[(g'_{e_1}\wedge 1 - g'_{e_1}) g''_{e_1}]
    + \EE[(g'_{e_1} - \tg'_{e_1}) g''_{e_1}].
  \end{equation}
  Using $0 \leq
  g''_{e_1} \leq 1$, \eqref{eq:over-control}, the Cauchy-Schwarz
  inequality, and \eqref{eq:lfc-nu'},
  we can bound the magnitude of the first term in
  \eqref{eq:disc-cap-split} by
  \[
  \EE[\abs{\nu'_{e_1} - 1}]
  \leq \EE[(\nu'_{e_1} - 1)^2]^{1/2} = o(1).
  \]
  The second term on the right-hand side of \eqref{eq:disc-cap-split}
  is equal to
  \[
      \EE\Bigl[\Bigl(\prod_{e \in H \setminus \{e_1\}} g_e(x_e) - \prod_{e
        \in H \setminus \{e_1\}}
        \tg_e(x_e)\Bigr) g''_{e_1}(x_{e_1})
        \Big\vert
        x \in V_J
        \Bigr].
  \]
  This is $o_{\e \to 0}(1)$ by the induction hypothesis
  applied to new weighted hypergraphs where the old
  $(\nu_{e_1}, g_{e_1},\tg_{e_1})$ gets replaced by $(1,g''_{e_1},g''_{e_1})$, thereby
  decreasing $\lvert\{ e \in H : \nu_e \neq 1\}\rvert$. Note that the
  linear forms condition continues to hold. Thus \eqref{eq:cap-disc}
  holds, so $(g'_{e_1}\wedge 1,
  \tg'_{e_1})$ is an $\e'$-discrepancy pair with $\e' = o_{\e\to 0}(1)$.
  \hfill$\square$

  \medskip

  We expand the second term of \eqref{eq:cutoff} as
  \begin{equation} \label{eq:densified-sum}
  \EE[(g'_{e_1} - \tg'_{e_1})(g'_{e_1} \wedge 1 - \tg'_{e_1})]
= \EE[g'_{e_1} (g'_{e_1}\wedge 1)] - \EE[g'_{e_1} \tg'_{e_1} ] -
\EE[\tg'_{e_1} (g'_{e_1}\wedge 1)] + \EE[(\tg'_{e_1})^2].
\end{equation}
We claim that each expectation on the right-hand side of
\eqref{eq:densified-sum} is $\EE[(\tg'_{e_1})^2] + o_{\e\to 0}(1)$. Indeed, by \eqref{eq:g'} and
  \eqref{eq:tg'} we have
\[
    \EE[g'_{e_1} (g'_{e_1}\wedge 1)]  - \EE[(\tg'_{e_1})^2] = \EE\Bigl[\Bigl((g'_{e_1}(x_{e_1}) \wedge 1) \hspace{-.8em}\prod_{e \in
      H \setminus \{e_1\}}\hspace{-.8em} g_e(x_e) \ - \ \tg'_{e_1}(x_{e_1}) \hspace{-.8em}\prod_{e \in H \setminus \{e_1\}}\hspace{-.8em}
    \tg_e(x_e)\Bigr)
    \Big\vert
    x \in V_J
    \Bigr],
\]
  which is $o_{\e \to 0}(1)$ by
  the induction hypothesis applied to new weighted hypergraphs where
  the old $(\nu_{e_1},g_{e_1},\tg_{e_1})$ is replaced by $(1, g'_{e_1}
  \wedge 1, \tg'_{e_1})$. This is allowed as $(g'_{e_1} \wedge 1,
  \tg'_{e_1})$ is an $\e'$-discrepancy pair with $\e' = o_{\e\to
    0}(1)$, the new $\nu$ still satisfies the linear forms condition, and $\abs{\set{e \in H : \nu_e \neq 1}}$ has decreased. The
  claims that the other terms on the right-hand side of
\eqref{eq:densified-sum} are each $\EE[(\tg'_{e_1})^2] + o_{\e\to 0}(1)$ are similar (in fact, easier). It follows that
  \eqref{eq:densified-sum} is $o_{\e \to 0}(1)$, so \eqref{eq:cutoff} is $o_{\e \to 0}(1)$ and
  we are done. \end{proof}

\section{Concluding remarks} \label{sec:concluding}

\noindent
\textbf{Conditions for counting lemmas.} In this paper, we determined sufficient conditions for establishing a relative Szemer\'edi theorem and, more generally, a counting lemma for sparse hypergraphs. We have assumed that the hypergraph we want to count within is a subgraph of a pseudorandom hypergraph. The main question then is to determine a good notion of pseudorandomness which is suffficient to establish a counting lemma.

There is a marked difference between this paper and our previous paper on graphs \cite{CFZ14} in terms of the type of pseudorandom condition assumed for the majorizing hypergraph. In this paper, we prove a counting lemma for a given hypergraph $H$ by assuming that the underlying pseudorandom hypergraph contains approximately the correct count for each hypergraph in a certain collection of hypergraphs $\mathcal{H}$ derived from $H$. That is, for each $H' \in \mathcal{H}$, we assume that our pseudorandom hypergraph contains $(1 + o(1)) p^{e(H')} n^{v(H')}$ labeled copies of $H'$, where $p$ is the edge density of the pseudorandom hypergraph.

The approach used in \cite{CFZ14} is equivalent, up to some polynomial loss in $\epsilon$, to assuming that the number of labeled cycles of length $4$ in our pseudorandom graph is $(1 + \epsilon)p^4 n^4$, where $\epsilon$ is now a carefully controlled term and the question of whether $H$ can be embedded in our pseudorandom graph depends on whether $\epsilon$ is sufficiently small with respect to $H$ and $p$. It is possible to adapt the methods of this paper so that the notion of pseudorandomness used for hypergraphs is more closely related to this latter notion. However, for the purposes of applying the results to a relative Szemer\'edi theorem, the current formulation seemed more appropriate.

\medskip

\noindent
\textbf{Gowers uniformity norms.}
For a function $f:\ZZ_N \rightarrow \mathbb{R}$, the Gowers $U^r$-norm of $f$ is defined to be
\[\norm{f}_{U^r}= \EE\Bigl[ \prod_{\omega \in \{0,1\}^r}
  f(x_0+\omega \cdot {\bf x})\Big\vert x_0,x_1,\ldots,x_r \in \ZZ_N \Bigr]^{1/2^{r}},
  \]
where ${\bf x}=(x_1,\ldots,x_r)$.
The following inequality, referred to as a generalized von Neumann theorem, bounds the weighted count of $(r+1)$-term arithmetic progressions from functions $f_0,\ldots,f_r$ in terms of the Gowers uniformity norm:
\[\Bigl\lvert \EE\Bigl[f_0(x)f_1(x+d)f_2(x+2d) \cdots f_r(x+rd)
\Big\vert x,d \in \mathbb{Z}_N \Bigr]\Bigr\rvert \leq \norm{f_j}_{U^r}
\prod_{i\neq j} \norm{f_i}_{\infty}.
\]
This fundamental fact is an important starting point for Gowers' celebrated proof \cite{Gow01}
of Szemer\'edi's theorem as well as many later developments in
additive combinatorics. For a sparse set $S \subseteq \ZZ_N$ of
density $p$, this inequality implies the correct count of $(r+1)$-term
arithmetic progressions in $S$ as long as $\norm{\nu-1}_{U^r} = o(p^r)$,
where $\nu=p^{-1}1_S$ (a more careful analysis shows that it suffices
to assume $\norm{\nu - 1}_{U^r} = o(p^{r/2})$).

Gowers~\cite{Gow10}\footnote{This question can be found in the penultimate
  paragraph in \S4 of the arXiv version of \cite{Gow10}.}
and Green~\cite{GreenPC} asked if  $\norm{\nu-1}_{U^{s}} = o(1)$
for some large $s = s(r)$ is sufficient for $\nu$ to satisfy a
relative Szemer\'edi theorem for $(r+1)$-term arithmetic
progressions. Note that this is precisely a linear forms condition and
we proved in this paper that a different linear forms condition is
sufficient. However, we do not even know if such a condition implies
the existence of $(r+1)$-term arithmetic progressions in
$\nu$. Clearly $s(r)$ cannot be too small and indeed we
know from the recent work of
Bennett and Bohman \cite{BB} on the random AP-free process that one
can find a $3$-AP-free $S \subset \ZZ_N$ such that $\nu = (N/|S|)1_S$ satisfies
$\|\nu-1\|_{U^2} = o(1)$. Therefore, if $s(2)$ exists, it must be
greater than $2$. More generally, they show that $s(r) > 1 + \log_2 r$. In a companion note \cite{CFZab}, we show that if a measure $\nu$ satisfies the stronger condition $\norm{\nu-1}_{U^r} = o(p^r)$, where $p=\norm{\nu}_{\infty}^{-1}$, then the relative Szemer\'edi theorem holds with respect to $\nu$ for $(r+1)$-term arithmetic
progressions. This strengthens the consequence of the generalized von Neumann theorem
discussed above.


\medskip

\noindent
\textbf{Corners in products of pseudorandom sets.}
Example~\ref{ex:corner} illustrates the
relative multidimensional Szemer\'edi theorem applied to a
pseudorandom set $S \subset
\ZZ_N^2$. However, the situation is quite different for $S \x S
\subset \ZZ_N^2$ with some
pseudorandom set $S \subset \ZZ_N$. Indeed, $S \x S \subset \ZZ_N^2$
does not satisfy the linear forms condition in
Example~\ref{ex:corner}. Intuitively, this is because the events
$(x,y) \in S \x S$ and $(x,y') \in S \x S$ are correlated as both
involve $x \in S$.

However, we may still deduce the following result using our
relative triangle removal lemma. Recall that a corner in $\ZZ_N^2$ is a
set of the form $\{(x,y),(x+d,y),(x,y+d)\}$, where $d \neq 0$.

\begin{proposition}
  \label{prop:product-corner}
If $S \subset \ZZ_N$ is
such that $\nu = \frac{N}{|S|}1_S$ satisfies
\begin{multline} \label{eq:prod-corner-lfc}
  \EE[\nu(x)\nu(x')\nu(z-x)\nu(z-x')\nu(z'-x)\nu(z'-x') \\
  \cdot \nu(y)\nu(y') \nu(z-y) \nu(z-y') \nu(z'-y) \nu(z'-y') \vert x,x',y,y',z,z'
  \in \ZZ_N] = 1 + o(1)
\end{multline}
and similar conditions hold if any subset of the $\nu$ factors are
erased, then any corner-free subset of $S \x S$ has size $o(|S|^2)$.
\end{proposition}

\begin{proof}[Proof (sketch)]
  Let $A$ be a corner-free subset of $S \x S$.  We build two
  tripartite graph $\Gamma$ and $G$ on the same vertex set $X \cup Y
  \cup Z$ with $X = Y = S$ and $Z = \ZZ_N$ (note that unlike the proof
  of Theorem~\ref{thm:rel-mul-sz} we do not take $X$ and $Y$ to be the
  whole of $\ZZ_N$ here). In $\Gamma$, we place a
  complete bipartite graph between $X$ and $Y$; between $Y$ and $Z$
  the edge $(y,z) \in Y \x Z$ is present if and only if $z-y \in S$; and 
  between $X$ and $Z$ the edge $(x,z) \in X \x Z$ is present if and
  only if $z-x \in S$. In $G$, between $X$ and $Y$ the edge $(x,y) \in
  (X,Y)$ is present if and only if $(x,y) \in A$; between $Y$ and $Z$
  the edge $(y,z) \in Y \x Z$ is present if and only if $(z-y,y) \in
  A$; and between $X$ and $Z$ the edge $(x,z) \in X \x Z$ is present if
  and only if $(x,z-x) \in A$.

  The vertices $(x,y,z) \in X \x Y \x Z$ form a triangle if and only
  if $(x,y),(z-y,y),(x,z-x) \in A$. These three points form a corner,
  which is degenerate only when $x+y =z$. Since $A$ is corner-free,
  every edge of $G$ is contained in exactly one triangle (namely
  the one that completes the equation $x + y = z$). In particular, $G$
  contains exactly $|A|$ triangles. After checking some
  hypotheses, we can apply our relative triangle removal lemma
  (as a special case of Theorem \ref{thm:rel-removal}) to conclude
  that it is possible to remove all triangles from $G$ by deleting $o(|S|^2)$
  edges. Since every edge of $G$ is contained in exactly one triangle, and $|G|$
  has $3|A|$ edges, we have $|A| = o(|S|^2)$, as
  desired.
\end{proof}

One can easily generalize the above Proposition to $S^m \subset
\ZZ_N^m$ (as before, $S \subset \ZZ_N$). Here a
corner is a set of the form $\{\bx, \bx + d\be_1, \dots, \bx + d\be_m\}$,
where $\bx \in \ZZ_N$, $0 \ne d \in \ZZ_N$, and $\be_i$ is the $i$-th
coordinate vector. Then, for any fixed $m$, any corner-free subset of
$S^m$ must have size $o(|S|^m)$, provided that $\nu = \frac{N}{|S|}1_S$
satisfies the linear forms condition
\begin{multline*}
\EE\biggl[ \prod_{i=1}^m \Bigl(\nu(x_i^{(0)})^{n_{i,0}}
\nu(x_i^{(1)})^{n_{i,1}} \hspace{-1.5em} \prod_{\omega
  \in \{0,1\}^{\{0\} \cup [m]\setminus \{i\}}} \hspace{-1.5em}  (x_0^{(\omega_0)} - \sum_{j \in [m]
    \setminus \{i\}} x_j^{(\omega_j)})^{n_{i,\omega}} \Bigr) \\
  \Big\vert x_0^{(0)},x_0^{(1)},\dots,x_m^{(0)},x_m^{(1)} \in \ZZ_N \biggr] = 1+o(1)
\end{multline*}
for any choices of exponents $n_{i,0},n_{i,1},n_{i,\omega} \in
\{0,1\}$.

A more general result concerning the existence of arbitrarily shaped
constellations in $S^m$ is known, provided that $S$ satisfies certain
stronger linear forms hypotheses. We refer the readers to
\cite{CMT13ar, FZ, TZ13ar} for further details. In particular, the
multidimensional relative Szemer\'edi theorem holds in
$P^m$, where $P$ is the primes.

\medskip

\noindent
\textbf{Sparse graph limits.}  The regularity method played a fundamental
role in the development of the theory of dense graph limits
\cite{BCLSV08, LS06}. However, no satisfactory theory of graph limits
is known for graphs with edge density $o(1)$. Bollob\'as and Riordan
\cite{BR09} asked a number of questions and made explicit conjectures on suitable
conditions for sparse graph limits and counting
lemmas. Our work gives some natural sufficient conditions for
obtaining a counting lemma in a
sequence of sparse graphs $G_N$. The new counting lemma allows us to transfer the results of Lov\'asz and Szegedy~\cite{LS06,LS07} on the
existence of the limit graphon, as well as the results of Borgs, Chayes, Lov\'asz,
S\'os, and Vesztergombi~\cite{BCLSV08} on the equivalence of
left-convergence (i.e., convergence in homomorphism densities) and
convergence in cut distance. The famous quasirandomness results of Chung,
Graham, and Wilson~\cite{CGW89} also transfer, namely, that an
appropriate relationship between edge density and
$C_4$-density (of homomorphisms) determines the asymptotic $F$-density for
every graph $F$. We will explain these connections in more detail in an
upcoming survey article \cite{CFZa}.

\medskip

\noindent
\textbf{Existing applications of the Green-Tao method.} Though our discussion has focused on the relative Szemer\'edi theorem, we have proved a relative version of the stronger multidimensional Szemer\'edi theorem. Following Tao \cite{Tao06jam}, this may be used to prove that the Gaussian primes contain arbitrarily shaped constellations, though without the need to verify either the correlation condition or the dual function condition. It seems likely that our method could also be useful for simplifying several other papers where the machinery of Green and Tao is used \cite{CM12, GT10, Le11, M12, M122, TZ08}. In some cases it should be possible to use our results verbatim but in others, such as the paper of Tao and Ziegler \cite{TZ08} proving that there are arbitrarily long polynomial progressions in the primes, it will probably require substantial additional work.




\medskip

\noindent
\textbf{Sparse hypergraph regularity.} In proving a hypergraph removal lemma for subgraphs of pseudorandom hypergraphs, we have developed a general approach to regularity and counting in sparse pseudorandom hypergraphs which has the potential for much broader application. It is, for example, quite easy to use our results to prove analogues of well-known combinatorial theorems such as Ramsey's theorem and Tur\'an's theorem relative to sparse pseudorandom hypergraphs of density $N^{-c_H}$. We omit the details. In the graph case, a number of further applications were discussed in \cite{CFZ14}. We expect that hypergraph versions of many of these applications should be an easy corollary of our results.
\medskip

 \noindent
{\bf Counting in random hypergraphs.} There has been much recent work on counting lemmas and relative versions of combinatorial theorems within random graphs and hypergraphs \cite{BMS, CG, CGSS, SaxT, Sch}. Surprisingly, there are a number of disparate approaches to these problems, each having its own strengths and weaknesses. We believe that our results can be used to give an alternative framework for one of these approaches, due to Conlon and Gowers \cite{CG}.\footnote{This should at least be true for theorems regarding graphs and hypergraphs, though we feel that a similar approach should also be possible for subsets of the integers.} Their proof relies heavily upon an application of the Green-Tao transference theorem, which we believe can be replaced with an application of the sparse Frieze-Kannan regularity lemma and our densification technique. However, the key technical step in \cite{CG}, which in our language is to verify that the strong linear forms condition, Lemma \ref{lem:strong-linear-forms}, holds when $\nu$ is a random measure, would remain unchanged.

\medskip

\noindent
\textbf{Sparse arithmetic removal.} In Theorem~\ref{thm:rel-mul-sz-removal}, we proved an arithmetic removal lemma for linear patterns such as arithmetic progressions. More generally, an arithmetic removal lemma
claims that if a system of linear equations $Ma = b$ over the integers has a small
number of solutions $a = (a_1, a_2, \dots, a_n)$ with $a_i \in A_i$
for all $i = 1, 2, \dots, n$ then one may remove a small number of
elements from each $A_i$ to find subsets $A'_i$ such that there are no
solutions $a' = (a'_1, a'_2, \dots, a'_n)$ to $Ma' = b$ with $a'_i \in
A'_i$ for all $i = 1, 2, \dots, n$. Such a result was conjectured by
Green \cite{G05gafa} and proved by Kr\'al', Serra, and Vena
\cite{KSV12} and, independently, Shapira \cite{Sha10}. Both of these
proofs are based upon representing a system of linear equations by a
hypergraph and deducing the arithmetic removal lemma from a hypergraph
removal lemma. Such an idea was first used by Kr\'al', Serra, and Vena
\cite{KSV09} with graphs (instead of hypergraphs). In \cite{CFZ14}, we
adapted the arguments of \cite{KSV09} to sparse pseudorandom subsets
of the integers using the removal lemma in sparse pseudorandom
graphs. Likewise, our results on hypergraph removal in this paper may
be used to prove a sparse pseudorandom generalization of the
arithmetic removal lemma \cite{KSV12,Sha10} for all systems of linear equations.

\medskip

\noindent \textbf{Acknowledgements.} We would like to thank Tom Bohman and Ben Green for helpful discussions.


\providecommand{\bysame}{\leavevmode\hbox to3em{\hrulefill}\thinspace}
\providecommand{\MR}{\relax\ifhmode\unskip\space\fi MR }
\providecommand{\MRhref}[2]{%
  \href{http://www.ams.org/mathscinet-getitem?mr=#1}{#2}
}
\providecommand{\href}[2]{#2}

\end{document}